\documentclass[leqno,11pt]{article}
\usepackage[utf8]{inputenc}
\usepackage[T1]{fontenc}
\usepackage{microtype}

\usepackage[letterpaper]{geometry}

\ifdefined\screenview
  \edef\mtht{\the\textheight}
  \edef\mtwd{\the\textwidth}
  \usepackage{xcolor}
  \definecolor{BackgroundColor}{RGB}{253, 246, 227}
  \pagecolor{BackgroundColor}
\fi

\usepackage{amsmath}
\usepackage{amsthm}
\usepackage{tikz-cd}
\usetikzlibrary{arrows} 
\tikzset{
  commutative diagrams/.cd, 
  arrow style=tikz, 
  diagrams={>=stealth}
}
\usetikzlibrary{matrix,decorations.pathreplacing,calc}

\usepackage{textcomp}

\usepackage[sb]{libertine}
\usepackage[varqu,varl]{zi4}
\usepackage[libertine,bigdelims,vvarbb]{newtxmath} 
\usepackage[supstfm=libertinesups,%
  supscaled=1.2,%
  raised=-.13em]{superiors}
\useosf 
\usepackage[scr=boondox,cal=euler]{mathalfa}

\usepackage{slashed}
\usepackage{esint} 
\usepackage[english]{babel}

\usepackage{imakeidx}
\makeindex[intoc]

\usepackage{csquotes}
\usepackage[
  backend=biber,
  hyperref=true,
  backref=true,
  isbn=false,
  doi=true,
  natbib=true,
  eprint=true,
  useprefix=true,
  maxcitenames=99,
  maxbibnames=99,  
  maxalphanames=99, 
  minalphanames=99,
  safeinputenc,
  style=alphabetic,
  citestyle=alphabetic,
  block=space,
  datamodel=preamble/ext-eprint
]{biblatex}
\usepackage[
  hypertexnames = false,
  colorlinks    = true,
  citecolor     = gray,
  linkcolor     = gray,
  urlcolor      = gray,
  breaklinks
]{hyperref}
\makeatletter
\DeclareFieldFormat{arxiv}{%
  arXiv\addcolon\space
  \ifhyperref
    {\href{http://arxiv.org/\abx@arxivpath/#1}{%
       \nolinkurl{#1}%
       \iffieldundef{arxivclass}
         {}
         {\addspace\texttt{\mkbibbrackets{\thefield{arxivclass}}}}}}
    {\nolinkurl{#1}
     \iffieldundef{arxivclass}
       {}
       {\addspace\texttt{\mkbibbrackets{\thefield{arxivclass}}}}}}
\makeatother
\DeclareFieldFormat{mr}{%
  MR\addcolon\space
  \ifhyperref
    {\href{http://www.ams.org/mathscinet-getitem?mr=MR#1}{\nolinkurl{#1}}}
    {\nolinkurl{#1}}}
\DeclareFieldFormat{zbl}{%
  Zbl\addcolon\space
  \ifhyperref
    {\href{http://zbmath.org/?q=an:#1}{\nolinkurl{#1}}}
    {\nolinkurl{#1}}}
\renewbibmacro*{eprint}{%
  \printfield{arxiv}%
  \newunit\newblock
  \printfield{mr}%
  \newunit\newblock
  \printfield{zbl}%
  \newunit\newblock
  \iffieldundef{eprinttype}
    {\printfield{eprint}}
    {\printfield[eprint:\strfield{eprinttype}]{eprint}}
}
\AtEveryBibitem{%
  \clearlist{publisher}%
}
\AtEveryBibitem{%
  \clearlist{address}%
}
\DeclareFieldFormat[article,inproceedings,inbook,incollection,thesis]{title}{\textit{#1}}
\renewbibmacro{in:}{}
\newcommand{\printreferences}{\printbibliography[heading=bibintoc]}

\usepackage[inline,shortlabels]{enumitem}

\usepackage{subcaption}

\usepackage[yyyymmdd]{datetime}

\usepackage{etoolbox}
\ifundef{\abstract}{}{\patchcmd{\abstract}%
    {\quotation}{\quotation\noindent\ignorespaces}{}{}}

\usepackage[super]{nth}

\usepackage{aliascnt}

\numberwithin{equation}{section}

\renewcommand{\eqref}[1]{\hyperref[#1]{\rm(\ref*{#1})}}

\def\makeautorefname#1#2{\AtBeginDocument{\expandafter\def\csname#1autorefname\endcsname{#2}}}

\newcommand{\mynewtheorem}[2]{
  \newaliascnt{#1}{equation}          
  \newtheorem{#1}[#1]{#2}
  \aliascntresetthe{#1}
  \makeautorefname{#1}{#2}
}

\mynewtheorem{axiom}{Axiom}
\newtheorem*{axiom*}{Axiom}
\mynewtheorem{theorem}{Theorem}
\newtheorem*{theorem*}{Theorem}
\mynewtheorem{prop}{Proposition}
\newtheorem*{prop*}{Proposition}

\mynewtheorem{cor}{Corollary}
\mynewtheorem{construction}{Construction}
\mynewtheorem{lemma}{Lemma}
\mynewtheorem{conjecture}{Conjecture}
\newtheorem*{conjecture*}{Conjecture}
\mynewtheorem{hyp}{Hypothesis}
\newtheorem{step}{Step}

\numberwithin{substep}{step}
\makeautorefname{step}{Step}
\makeautorefname{substep}{Step}
\newtheorem{case}{Case}
\newtheorem{subcase}{Case}
\numberwithin{subcase}{case}
\makeautorefname{case}{Case}
\makeautorefname{subcase}{case}

\usepackage{etoolbox}
\AtBeginEnvironment{proof}{\setcounter{step}{0}}

\theoremstyle{remark}
\mynewtheorem{remark}{Remark}
\newtheorem*{remark*}{Remark}
\mynewtheorem{convention}{Convention}
\newtheorem*{convention*}{Convention}
\newtheorem*{conventions*}{Conventions}

\theoremstyle{remark}
\mynewtheorem{warning}{Warning}

\theoremstyle{definition}
\mynewtheorem{definition}{Definition}
\newtheorem*{definition*}{Definition}
\mynewtheorem{notation}{Notation}
\mynewtheorem{data}{Data}
\mynewtheorem{example}{Example}
\newtheorem*{example*}{Example}
\mynewtheorem{exercise}{Exercise}
\mynewtheorem{solution}{Solution}
\mynewtheorem{question}{Question}
\mynewtheorem{problem}{Problem}
\newtheorem*{question*}{Question}

\mynewtheorem{summary}{Summary}

\makeautorefname{table}{Table}        
\makeautorefname{chapter}{Chapter}
\makeautorefname{section}{Section}
\makeautorefname{subsection}{Section}
\makeautorefname{subsubsection}{Section}
\makeautorefname{footnote}{Footnote}
\AtBeginDocument{\def\itemautorefname~#1\null{(#1)\null}}
\AtBeginDocument{\def\equationautorefname~#1\null{(#1)\null}}

\let\C\undefined

\usepackage{bm}
\usepackage{mathtools} 
\usepackage{stmaryrd} 

\DeclareFontFamily{U}{mathx}{\hyphenchar\font45}
\DeclareFontShape{U}{mathx}{m}{n}{
      <5> <6> <7> <8> <9> <10>
      <10.95> <12> <14.4> <17.28> <20.74> <24.88>
      mathx10
      }{}
\DeclareSymbolFont{mathx}{U}{mathx}{m}{n}
\DeclareFontSubstitution{U}{mathx}{m}{n}
\DeclareMathAccent{\widecheck}{0}{mathx}{"71}
\DeclareMathAccent{\wideparen}{0}{mathx}{"75}

\DeclareMathOperator{\End}{End}

\DeclareMathOperator{\Fr}{Fr}
\DeclareMathOperator{\GL}{GL}

\DeclareMathOperator{\HF}{\HF}

\DeclareMathOperator{\Hol}{Hol}
\DeclareMathOperator{\Hom}{Hom}

\DeclarePairedDelimiter{\set}{\lbrace}{\rbrace}
\def\({\left(}
\def\){\right)}
\def\<{\left\langle}
\def\>{\right\rangle}

\newcommand{\C}{{\mathbf{C}}}
\newcommand{\Gtwo}{G_2}

\newcommand{\N}{{\mathbf{N}}}

\newcommand{\R}{\mathbf{R}}

\newcommand{\SO}{\mathrm{SO}}
\newcommand{\SU}{\mathrm{SU}}

\newcommand{\Spin}{\mathrm{Spin}}

\newcommand{\Z}{\mathbf{Z}}

\newcommand{\andq}{\text{and}\quad}

\newcommand{\co}{\mskip0.5mu\colon\thinspace}

\newcommand{\del}{\partial}

\newcommand{\id}{\mathrm{id}}

\newcommand{\iso}{\cong}

\newcommand{\loc}{\mathrm{loc}}

\newcommand{\qandq}{\quad\text{and}\quad}

\newcommand{\qand}{\quad\text{and}}

\newcommand{\sing}{\mathrm{sing}}

\newcommand{\so}{\mathfrak{so}}

\newcommand{\vol}{\mathrm{vol}}

\renewcommand{\H}{\mathbf{H}}
\renewcommand{\Im}{\operatorname{Im}}

\renewcommand{\Re}{\operatorname{Re}}

\renewcommand{\epsilon}{\varepsilon}
\renewcommand{\setminus}{{\backslash}}

\renewcommand{\leq}{\leqslant}
\renewcommand{\geq}{\geqslant}

\makeatletter
\renewcommand*\env@matrix[1][*\c@MaxMatrixCols c]{%
  \hskip -\arraycolsep
  \let\@ifnextchar\new@ifnextchar
  \array{#1}}

\renewcommand\xleftrightarrow[2][]{%
  \ext@arrow 9999{\longleftrightarrowfill@}{#1}{#2}}
\newcommand\longleftrightarrowfill@{%
  \arrowfill@\leftarrow\relbar\rightarrow}
\makeatother

\newcommand{\rd}{{\rm d}}

\newcommand{\rI}{{\rm I}}
\newcommand{\rII}{{\rm II}}
\newcommand{\rIII}{{\rm III}}

\newcommand{\ua}{{\underline a}}
\newcommand{\ub}{{\underline b}}

\newcommand{\bA}{{\mathbf{A}}}

\newcommand{\bE}{{\mathbf{E}}}

\newcommand{\bL}{{\mathbf{L}}}

\newcommand{\sA}{\mathscr{A}}
\newcommand{\sB}{\mathscr{B}}

\newcommand{\sG}{\mathscr{G}}

\newcommand{\sM}{\mathscr{M}}
\newcommand{\sN}{\mathscr{N}}

\newcommand{\sP}{\mathscr{P}}

\newcommand{\fg}{{\mathfrak g}}
\newcommand{\frg}{{\mathfrak g}}

\newcommand{\fs}{{\mathfrak s}}

\newcommand{\fu}{{\mathfrak u}}

\newcommand{\fA}{{\mathfrak A}}

\newcommand{\fF}{{\mathfrak F}}

\newcommand{\fI}{{\mathfrak I}}

\newcommand{\fM}{{\mathfrak M}}

\newcommand{\fX}{{\mathfrak X}}
\newcommand{\fY}{{\mathfrak Y}}

\newcommand{\slD}{\slashed D}
\newcommand{\slS}{\slashed S}

\author{
  Thomas Walpuski
}
\title{$\Gtwo$--instantons, associative submanifolds, and
  Fueter~sections}
\date{2016-04-29}

\begin{document}

\maketitle

\begin{abstract}
  We give sufficient conditions for a family of $\Gtwo$--instantons to be ``spontaneously'' be born out of a Fueter section of a bundle of moduli space of ASD instantons over an associative submanifold.
  This phenomenon is one of the key difficulties in defining the conjectural \emph{$\Gtwo$ Casson invariant} proposed by Donaldson--Thomas in \cite{Donaldson1998}.
\end{abstract}

\section{Introduction}
\label{Sec_Introduction}

Fix a compact $7$--manifold $Y$ together with a positive $3$--form $\phi$ satisfying a certain non-linear partial differential equation; see \eqref{Eq_TorsionFree} and the discussion preceding it.
The $3$--form $\phi$ canonically equips $Y$ with a metric (and orientation) such that the holonomy group $\Hol(g)$ is contained in the exceptional Lie group $\Gtwo$;
hence, $(Y,\phi)$ is commonly called a $\Gtwo$--manifold and $\phi$ is called a torsion-free $\Gtwo$--structure.

Given a $G$--bundle $E$ over $Y$, Donaldson and Thomas \cite{Donaldson1998} noted that there is a Chern--Simons type functional on $\sB(E)$, the space of gauge equivalence classes of connections, whose critical points $[A]$ satisfy
\begin{equation}
  \label{Eq_G2Instanton}
  *(F_A\wedge\phi) = - F_A.
\end{equation}
Solutions of \eqref{Eq_G2Instanton} are called $\Gtwo$--instantons.
These are the central objects in gauge theory on $\Gtwo$--manifolds.
The moduli space of $\Gtwo$--instantons
\begin{equation*}
  \sM(E,\phi):=\set{ [A]\in\sB(E) : *(F_A\wedge\phi)=-F_A }
\end{equation*}
can, in general, be a very complicated space.
However, after gauge fixing, \eqref{Eq_G2Instanton} has an elliptic deformation theory of index zero, i.e., $\sM(E,\phi)$ has virtual dimension zero.
Thus one can try to ``count'' $\sM(E,\phi)$, say, by a suitable perturbation scheme or via virtual cycle techniques and arrive at a number
\begin{equation*}
  n(E,\phi):=\#\sM(E,\phi).
\end{equation*}  

How does $n(E,\phi)$ depend on $\phi$?
Since the deformation theory of $\Gtwo$--instantons is very well-behaved, the key question one needs to understand is: how can $\Gtwo$--instantons degenerate as $\phi$ varies?
Consider a family of $\Gtwo$--instantons $(A_t)_{t \in (0,T]}$ over a family of $\Gtwo$--manifolds $(Y,\phi_t)_{t \in (0,T]}$ and assume that $\phi_t$ converges to a torsion-free $\Gtwo$--structure $\phi_0$ as $t \to 0$.
From classical results due to Uhlenbeck \cite{Uhlenbeck1982a}, Price \cite{Price1983} and Nakajima \cite{Nakajima1988} and more recent progress by Tian \cite{Tian2000} and Tao and Tian \cite{Tao2004} one can conclude the following:
\begin{itemize}
\item 
  There is a closed subset $P$ of $Y$ of finite $3$--dimensional Hausdorff measure and a $\Gtwo$--instanton $B$ over $(Y\setminus P,\phi_0)$ such that up to gauge transformations a subsequence of $(A_t)$ converges to $B$ in $C^\infty_\loc$ on $Y\setminus P$ as $t \to 0$.
\item
  $B$ can be extended to the complement of a closed set $\sing(B)$ of vanishing $3$--dimensional Hausdorff measure.
  However, $\sing(B)$ might very well be non-empty, that is: one might encounter non-removable singularities.
\item
  $P$ supports an integral current calibrated by $\phi_0$, or more informally: $P$ is a, possibly wildly singular, associative submanifold in $(Y,\phi_0)$.
  At almost every point $x \in P$, the degeneration of $(A_t)$ is modelled on (a bubbling tree of) ASD instantons bubbling off in the direction transverse to $P$.
\end{itemize}

This, of course, represents the worse case scenario.
One would expect that a generic deformation is less wild.
In this article we only consider the case when $B$ extends to all of $Y$ and $P$ is smooth.
Moreover, we form a bundle $\fM$ over $P$ whose fibres are moduli spaces of ASD instantons, as explained in \autoref{Sec_Fueter}, and assume that the ASD instantons bubbling off transverse to $P$ give rise to a section $\fI \in \Gamma(\fM)$.
Since the ASD instantons bubbling off do not have a canonical scale, $\fI$ is unique only up to the action of $C^\infty(P,\R_{>0})$.
Donaldson and Segal \cite{Donaldson2009} noticed that $\fI$ cannot be arbitrary but should satisfy a non-linear p.d.e.~called the Fueter equation, provided scalings are chosen appropriately; see \autoref{Sec_Fueter} for more details. 
This equation is elliptic of index zero;
however, since $\fM$ is a bundle of cones, one only expects solutions to appear only at isolated values in $1$--parameter families.
In particular, for a generic torsion-free $\Gtwo$--structure $\phi$, no Fueter section $\fI \in \Gamma(\fM)$ ought to exists.
Hence, the bubbling phenomenon for $\Gtwo$--instantons should only occur in codimension one.

The main result of this article is to prove that given the data $(B,P,\fI)$ and assuming certain ``acyclicity/unobstructedness conditions'' (which are expounded in \autoref{Def_G2InstantonAcyclic}, \autoref{Def_AssociativeUnobstructed} and \autoref{Def_FueterUnobstructed}), we can produce a family of $\Gtwo$--instantons yielding $(B,P,\fI)$ in the limit.

\begin{theorem}\label{Thm_A}
  Let $Y$ be a compact $7$--manifold equipped with a family of torsion-free $\Gtwo$--structures $(\phi_t)_{t\in(-T,T)}$.
  Suppose we are given:
  \begin{itemize}
  \item
    an acyclic $\Gtwo$--instanton $B$ on a $G$--bundle $E_0$ over $(Y,\phi_t)$,
  \item
    an unobstructed associative submanifold $P$ in $(Y,\phi_0)$ and
  \item
    a Fueter section $\fI$ of an instanton moduli bundle $\fM$ over $P$ associated $E_0|_{P}$ which is unobstructed with respect to $(\phi_t)$.
  \end{itemize}
  Then there is a constant $\Lambda>0$, a $G$--bundle $E$ together with a family of connections $(A_\lambda)_{\lambda\in(0,\Lambda]}$ and a continuous function $t\co[0,\Lambda]\to(-T,T)$ with $t(0) = 0$ such that:
  \begin{itemize}
  \item
    $A_\lambda$ is a $\Gtwo$--instanton on $E$ over $(Y,\phi_{t(\lambda)})$ for all $\lambda\in(0,\Lambda]$.
  \item 
     $A_\lambda$ converges to $B$ on the complement of $P$ and at each point $x\in P$ an ASD instanton in the equivalence class given by $\fI(x)$ bubbles off transversely as $\lambda \to 0$.
  \end{itemize}
\end{theorem}

As was already pointed out by Donaldson and Segal \cite{Donaldson2009}, an immediate consequence is that $n(E,\phi)$ has no reason to be invariant under (large) deformations of $\phi$.
They suggest that one should try to construct a counter term, say $m(E,\phi)$, as a weighted count of associative submanifolds and $\Gtwo$--instantons on bundles of ``smaller'' topological type than $E$, so that the sum $n(E,\phi)+m(E,\phi)$ is invariant under deformations.
The crucial point is to find out what these weights should be.
A candidate for the definition of these weights in the ``low energy'' $\SU(2)$--theory, which was hinted at by Donaldson--Segal \cite{Donaldson2009}, is explained in more detail in the author's PhD thesis \cite[Chapter 6]{Walpuski2013}.
A more systematic approach, based on generalised Seiberg--Witten equations and the ADHM construction, is currently being developed by Haydys and the author; see \cite{Haydys2014} for a first step.

\begin{remark}
  It would be interesting to see a concrete example of the input required by \autoref{Thm_A}.
  Unfortunately, no such example is known currently.
  The main difficulty with constructing such examples is to ensure that $\fI$ is unobstructed with respect to $(\phi_t)$. 
  It should be pointed out, however, that a construction closely related to \autoref{Thm_A} has been used by the author to construct $\Spin(7)$--instantons from a Fueter section of a bundle of moduli space of ASD instantons over a Cayley submanifold \cite{Walpuski2014}.
\end{remark}

\begin{remark}
  The proof of \autoref{Thm_A} is based on a gluing construction and the analysis involved is an extension of that required for the construction of $\Gtwo$--instantons on generalised Kummer constructions in \cite{Walpuski2011}.
  As such there are some similarities with Lewis' construction of $\Spin(7)$--instantons \cite{Lewis1998}, unpublished work by Brendle on $\Spin(7)$--instantons \cite{Brendle2003a} and Pacard--Ritoré's work on the Allen--Cahn equation \cite{Pacard2003}.
\end{remark}

\paragraph{Acknowledgements.}
This article is the outcome of work undertaken by the author for his PhD thesis at Imperial College London, supported by European Research Council Grant 247331.
I am grateful to my supervisor Simon Donaldson for his encouragement and for sharing some of his ideas with me.


\section{Review of geometry on \texorpdfstring{$\Gtwo$}{G2}--manifolds}
\label{Sec_G2Geometry}

We begin with a terse review of the basic notions of $\Gtwo$--geometry.
This is mainly to fix notation and conventions, and also recall a few results which we will make use of later.
The reader who is interested in a more detailed exposition is referred to Joyce's book \cite{Joyce2000}, which is the standard reference for most of the material in this section.

\begin{definition}
  \label{Def_Positive}
  A $3$--form $\phi$ on a $7$--dimensional vector space is called \emph{positive} if for each non-zero vector $v \in V$ the $2$--form $i(v)\phi$ on $V/\<v\>$ is symplectic.
\end{definition}

\begin{example}
  The $3$--form $\phi_0 \in \Omega^3(\R^7)$ defined by
  \begin{equation}
    \label{Eq_phi0}
    \phi_0 := \rd x^{123}-\rd x^{145}-\rd x^{167}-\rd x^{246}+\rd x^{257}-\rd x^{347}-\rd x^{356}
  \end{equation}
  is positive.
\end{example}

This example is representative in the sense that for any positive $3$--form $\phi$ on $V$ there exists a basis of $V$ with respect to which $\phi$ is given by $\phi_0$;
see, e.g., \cite[Theorem~3.2]{Salamon2010}.
Hence, the space of positive $3$--forms on $V$ is a $\GL(V)$--orbit.
The stabiliser of a fixed positive $3$--form is isomorphic to the exceptional Lie group $\Gtwo$.
The choice of a positive $3$--form $\phi$ equips $V$ with a canonical metric $g$ and orientation on $V$ such that
\begin{equation}
  \label{Eq_phiMetric}
  i(v_1)\phi\wedge i(v_2)\phi\wedge\phi=6g(v_1,v_2) \vol.
\end{equation}
In particular, if $\sP(V)$ denotes the space of positive $3$--forms on $V$, then there is a \emph{non-linear} map $\Theta\co \sP(V) \to \Lambda^4 V^*$ defined by
\begin{equation*}
  \Theta(\phi) := *_\phi \phi.
\end{equation*}

\begin{definition}
  A \emph{$\Gtwo$--structure} on a $7$--manifold $Y$ is a positive $3$--form $\phi \in \Gamma(\sP(TY)) \subset \Omega^3(Y)$.
  It is called \emph{torsion-free} if
  \begin{equation}
    \label{Eq_TorsionFree}
    \rd \phi = 0 \qandq \rd \Theta(\phi) = 0.
  \end{equation}
  A $7$--manifold $Y$ equipped with a torsion-free $\Gtwo$--structure $\phi$ is called a \emph{$\Gtwo$--manifold}.
\end{definition}

\begin{remark}
  From the above discussion is clear that a $\Gtwo$--structure is equivalent to a reduction of the structure group of the tangent bundle from $\GL(7)$ to $\Gtwo$.
  A theorem of Fernández and Gray \cite[Theorem 5.2]{Fernandez1982} asserts that \eqref{Eq_TorsionFree} is equivalent to $\nabla_g \phi = 0$;
  hence, for a torsion-free $\Gtwo$--structure, the holonomy group $\Hol(g)$ is contained in $\Gtwo$. 
\end{remark}

Examples of $\Gtwo$--manifold with $\Hol(g)$ strictly contained in $\Gtwo$ are easy to come by.
For our purposes the following very trivial example will play an important rôle.
\begin{example}
  \label{Ex_R3R4}
  Choose coordinates $\(x^1,x^2,x^3,y^1,\ldots,y^4\)$ on $\R^7=\R^3\oplus \R^4$ and set
  \begin{equation*}
    \omega_1 = \rd y^{12} + \rd y^{34}, \quad
    \omega_2 = \rd y^{13} - \rd y^{24} \qandq
    \omega_3 = \rd y^{14} + \rd y^{23}.
  \end{equation*}
  Then
  \begin{equation*}
    \phi = \rd x^{123}
    - \rd x^1 \wedge \omega_1
    - \rd x^2 \wedge \omega_2
    - \rd x^3 \wedge \omega_3
  \end{equation*}
  is a torsion-free $\Gtwo$--structure on $\R^7$.
\end{example}

There is by now a plethora of examples of $\Gtwo$--manifolds due to Bryant \cite{Bryant1987}, Bryant and Salamon \cite{Bryant1989}, Joyce \cite{Joyce1996},  Kovalev \cite{Kovalev2003}, Kovalev and Lee \cite{Kovalev2011}, and Corti, Haskins, Nordström and Pacini \cite{Corti2012a}.
The construction techniques (especially in the latter cases, which yield compact examples) are quite involved and we will not go into any detail.

\subsection{Gauge theory on \texorpdfstring{$\Gtwo$}{G2}--manifolds}
\label{Sec_G2GaugeTheory}

Let $(Y,\phi)$ be a \emph{compact} $\Gtwo$--manifold and let $E$ be a $G$--bundle over $Y$ where $G$ is a compact Lie group, say $G=\SO(3)$ or $G=\SU(2)$.
Denote by $\sA(E)$ the space of connections on $E$.

\begin{definition}
  A connection $A\in\sA(E)$ on $E$ is called a
  \emph{$\Gtwo$--instanton} on $(Y,\phi)$ if it satisfies \eqref{Eq_G2Instanton}, i.e.,
  \begin{equation*}
    *(F_A\wedge\phi) = - F_A.
  \end{equation*}
\end{definition}
Since $\phi$ is closed, it follows from the Bianchi identity that $\Gtwo$--instantons are Yang--Mills connections.
In fact, there is an energy identity which shows that $\Gtwo$--instantons are absolute minima of the Yang--Mills functional.

\begin{example}
  The pullback of an ASD instanton over $\R^4$ to $\R^7=\R^3\oplus\R^4$, as in \autoref{Ex_R3R4}, is a $\Gtwo$--instanton.
\end{example}

The first non-trivial examples of $\Gtwo$--instantons (with structure group $G=\SO(3)$) where recently constructed by the author in \cite{Walpuski2011}.
Those live on manifolds arising from Joyce's generalised Kummer construction.
A method to produce $\Gtwo$--instantons on $\Gtwo$--manifolds arising from the twisted connected sum construction was presented by Sá Earp and the author in \cite{SaEarp2013} and used to produce concrete examples by the author in \cite{Walpuski2015}.

From an analytical point of view equation~\eqref{Eq_G2Instanton} is slightly inconvenient to work with, because its linearisation supplemented with the Coulomb gauge is not elliptic.
However, we can make use of the following result whose proof can be found, e.g., in \cite[Proposition~3.7]{Walpuski2011}.

\begin{prop}
  \label{Prop_G2Instanton}
  Set $\psi : = \Theta(\phi)$. 
  Let $A\in\sA(E)$ be a connection on $E$.
  Then the following are equivalent.
  \begin{enumerate}
  \item
    $A$ is $\Gtwo$--instanton.
  \item
    $A$ satisfies $F_A\wedge\psi=0$.
  \item
    There is a $\xi\in\Omega^0(Y,\frg_E)$ such that
    \begin{align}\label{Eq_G2InstantonHiggs}
      F_A\wedge\psi+*\rd_A\xi=0.
    \end{align}
  \end{enumerate}
\end{prop}

From \autoref{Prop_G2Instanton} one can see that $\Gtwo$--instantons are in many ways similar to flat connections on $3$--manifolds.
In particular, if $A_0\in\sA(E)$ is a $\Gtwo$--instanton, then there is a
$\Gtwo$--Chern--Simons functional
\begin{equation*}
  CS^\psi(A_0+a):=\int_Y \<a\wedge\rd_{A_0} a + \frac13 a\wedge
  [a\wedge a]\>\wedge\psi
\end{equation*}
whose critical points are precisely the $\Gtwo$--instantons on $E$.
Very roughly speaking the conjectural $\Gtwo$ Casson invariant, suggested by Donaldson and Thomas \cite{Donaldson1998}, should be a signed count the critical points of $CS^\psi$ on a suitable completion of $\sA(E)/\sG(E)$.  Here $\sG(E)$ denotes the group of gauge transformations of $E$.

The infinitesimal deformation theory of $\Gtwo$--instantons around $A \in \sA(E)$ is governed by the self-dual elliptic complex
\begin{equation}
  \label{Eq_DeformationComplex}
  \Omega^0(Y,\fg_E) \xrightarrow{\rd_A}
  \Omega^1(Y,\fg_E) \xrightarrow{\psi\wedge\rd_A}
  \Omega^6(Y,\fg_E) \xrightarrow{\rd_A}
  \Omega^7(Y,\fg_E).
\end{equation}

\begin{definition}
  \label{Def_G2InstantonAcyclic}
  A $\Gtwo$--instanton $A \in \sA(E)$ is called \emph{irreducible}, \emph{rigid} or \emph{unobstructed} if \eqref{Eq_DeformationComplex} has vanishing cohomology in degree zero, one or two respectively.
  It is called \emph{acyclic} if the cohomology of \eqref{Eq_DeformationComplex} vanishes completely.
\end{definition}

\begin{remark}
  Since \eqref{Eq_DeformationComplex} is self-dual, rigid and unobstructed are the same thing;
  in particular, $A$ is acyclic if and only if it is irreducible and rigid/unobstructed.
\end{remark}

For any $A\in\sA(E)$ we define $L_A=L_{A,\phi}\co\Omega^0(Y,\frg_E)\oplus\Omega^1(Y,\frg_E) \to
\Omega^0(Y,\frg_E)\oplus\Omega^1(Y,\frg_E)$ by
\begin{align}
  \label{Eq_LAphi}
  L_{A,\phi}:=\begin{pmatrix}
    0 & \rd_A^* \\
    \rd_A & *\left(\psi\wedge\rd_A\right)
  \end{pmatrix}
\end{align}
where $\psi:=\Theta(\phi)$.
This is a self-adjoint elliptic operator.
It appears as the linearisation of equation~\eqref{Eq_G2InstantonHiggs} supplemented with the Coulomb gauge and therefore controls the infinitesimal deformation theory of $\Gtwo$--instantons.
Alternatively, $L_A$ is obtained by folding the complex \eqref{Eq_DeformationComplex}.

As an immediate consequence of the implicit function theorem we have the following result.

\begin{prop}
  \label{Prop_G2InstantonDeformation}
  Let $Y$ be a compact $7$--manifold and let $(\phi_t)_{t\in(-T,T)}$ be a family of torsion-free $\Gtwo$--structures on $Y$.
  Suppose that $A\in\sA(E)$ is an unobstructed $\Gtwo$--instanton on a
  $G$--bundle $E$ over $(Y,\phi_0)$.
  Then there is a constant $T'\in(0,T]$ and a unique family of $\Gtwo$--instantons $(A_t)_{t\in(-T',T')}$ on $E$ over $(Y,\phi_t)$ with $A_0=A$.
\end{prop}

\subsection{Associative submanifolds in \texorpdfstring{$\Gtwo$}{G2}--manifolds}
\label{Sec_Associatives}

Let $(Y,\phi)$ be a compact $\Gtwo$--manifold.
The $3$--form $\phi$ is a \emph{calibration} in the sense of Harvey--Lawson \cite{Harvey1982}, meaning that $\phi$ is closed and that for each oriented $3$--dimensional subspace $P$ of $T_xY$ the following inequality holds
\begin{equation*}
  \vol_{P} \leq \phi|_P.
\end{equation*}

\begin{definition}
  An oriented submanifold $P$ of $Y$ is called an \emph{associative submanifold} in $(Y,\phi)$ if it is calibrated by $\phi$, that is, for each $x\in P$ we have
  \begin{equation*}
    \vol_{T_xP}=\phi|_{T_xP}.
  \end{equation*}
\end{definition}

\begin{example}
  $\R^3\times\{0\}\subset \R^3\oplus \R^4$, as in \autoref{Ex_R3R4}, is an associative submanifold.
\end{example}

Associative submanifolds also arise as $3$--dimensional fixed point sets of orientation reversing involutions of $Y$ mapping $\phi$ to $-\phi$.
For concrete examples we refer the reader to Joyce \cite[Part II, Section 4.2]{Joyce1996}.
The recent work of Corti, Haskins, Nordström and Pacini \cite{Corti2012a} gives a number of concrete examples of associative submanifolds in twisted connected sums.

The importance of associative submanifolds in the study of gauge theory on $\Gtwo$--manifolds is due to the following fact:
Consider $(\R^7,\phi_0)$ and \emph{any} orthogonal decomposition $\R^7=\R^3\oplus \R^4$.
Let $I$ be a connection on a bundle over $\R^4$.
Then the pullback of $I$ to $\R^7$ is a $\Gtwo$--instanton if and only if  there is an orientation on $\R^3$ with respect to which it is calibrated by $\phi_0$ and $I$ is an ASD instanton on $\R^4$.
This is the underlying reason why the bubbling locus of a sequence of $\Gtwo$--instantons is associative and why the connections bubbling off transversely are ASD instantons.

In the following we will discuss some results due to McLean \cite{McLean1998} concerning the deformation theory of associative submanifolds.
If $P$ is an associative submanifold, then there is a natural identification
\begin{equation}
  \label{Eq_Clifford}
  TP \iso \Lambda^+ N^* P\co v \mapsto -i(v)\phi
\end{equation}
given by (the negative of) inserting tangent vectors to $P$ into $\phi$ (and restricting to $NP$).
Thinking of $\Lambda^+ N^* P$ as a sub-bundle of $\so(NP)$ this yields a Clifford multiplication $\gamma\co TP\to\End(NP)$.
Denote by $\bar\nabla$ the connection on $NP$ induced by the Levi--Civita connection on $Y$.
\begin{definition}
  The \emph{Fueter operator} $F_P=F_{P,\phi}\co\Gamma(NP)\to\Gamma(NP)$  associated with $P$ is defined by
  \begin{equation}
    \label{Eq_FPphi}
    F_{P,\phi} (n):=\sum_{i=1}^3 \gamma(e_i) \bar\nabla_i n
  \end{equation}
  with $(e_i)$ a local orthonormal frame on $P$.
\end{definition}

\begin{remark}
  \label{Rmk_FueterDirac}
  The Fueter operator $F_P$ can be identified with a twisted Dirac operator as follows.
  Pick a spin structure $\fs$ on $P$.
  Because of the identification~\eqref{Eq_Clifford} there is a unique $\SU(2)$--bundle $\fu$ over $P$ such that $\fs\times\fu$ is a spin structure on $NP$.
  The bundle $\fu$ also comes with a connection, such that the resulting connection on $\fs\times\fu$ is a spin connection.
  If $\slS$ and $U$ denote the quaternionic line bundles corresponding to $\fs$ and $\fu$, then $\slS\otimes_\C U$ has a natural real structure and its real part can be identified with $NP$.
  With respect to this identification $F_P$ becomes the twisted Dirac operator $\slD\co\Gamma(\Re(\slS\otimes_\C U))\to\Gamma(\Re(\slS\otimes_\C U))$.
\end{remark}

The importance of $F_P$ is that it controls the infinitesimal deformation theory of the associative submanifold $P$.
In particular, the moduli space of associative submanifolds near $P$ is modelled on the zero set of a smooth map from a neighbourhood of zero in the kernel of $F_P$ to its cokernel.

\begin{definition}
  \label{Def_AssociativeUnobstructed}
  An associative submanifold $P$ is called \emph{rigid} (\emph{unobstructed}) if $F_P$ is  injective (surjective).
\end{definition}

\begin{remark}
  Since $F_P$ is self-adjoint, unobstructed and rigid are the same thing.
  So unobstructed associative submanifolds are also rigid.
\end{remark}

Using McLean's setup for the deformation theory of associative submanifolds developed in \cite{McLean1998} the following is a simple consequence of the implicit function theorem.

\begin{prop}
  \label{Prop_AssociativeDeformation}
  Let $Y$ be a compact $7$--manifold and let $(\phi_t)_{t\in(-T,T)}$ be a family of torsion-free $\Gtwo$--structures on $Y$.
  Suppose that $P$ is an unobstructed associative submanifold in $(Y,\phi_0)$.
  Then there is a constant $T'\in(0,T]$ and a unique family of associative submanifolds $(P_t)_{t\in(-T',T')}$ in $(Y,\phi_t)$ with $P_0=P$.
\end{prop}


\section{Moduli spaces of ASD instantons over \texorpdfstring{$\R^4$}{R4}} 
\label{Sec_ASDModuli}

In the next section we will explain the construction of the bundle $\fM$ of moduli spaces of ASD instantons, the Fueter equation and provide more detail for the discussion preceding \autoref{Thm_A}.
As a preparation we quickly recall some basic facts about moduli spaces of ASD instantons over $\R^4$.

Fix a $G$--bundle $E$ over $S^4=\R^4\cup\{\infty\}$.
Denote by $M$ the moduli space of ASD instantons on $E$ framed over the point at infinity.
These moduli spaces are smooth manifolds, because ASD instantons over $S^4$ are always unobstructed as a consequence of the Weitzenböck formula; see, e.g., \cite[Proposition~2.2]{Taubes1982}.
By Uhlenbeck's removable singularities theorem \cite[Theorem~4.1]{Uhlenbeck1982} we can think of $M$ as a moduli space of framed finite energy ASD instantons on $\R^4$.
In a suitable functional analytic setup incorporating decay conditions at infinity, see, e.g., \cite{Taubes1983} or \cite{Nakajima1990}, the infinitesimal deformation theory of a framed ASD instanton $I$ over $\R^4$ is governed by the linear operator $\delta_I \co \Omega^1(\R^4,\fg_E) \to \Omega^0(\R^4,\fg_E) \oplus \Omega^+(\R^4,\fg_E)$ defined by
\begin{equation}
  \label{Eq_deltaI}
  \delta_I a:=(\rd_I^*a,\rd_I^+a).
\end{equation}
From the work of Taubes \cite{Taubes1983} it is known that $\delta_I$ is always surjective and that its kernel lies in $L^2$.
More precisely, we have the following result whose proof can be found, e.g., in \cite[Proposition~5.10]{Walpuski2011}.

\begin{prop}
  \label{Prop_Decay}
  Let $E$ be a $G$--bundle over $\R^4$ and let $I\in\sA(E)$ be a finite energy ASD instanton on $E$.
  Then the following holds.
  \begin{enumerate}
  \item
    If $a\in\ker\delta_I$ decays to zero at infinity, that is to say $\lim_{r\to\infty} \sup_{\del B_r(0)} |a|=0$, then $|\nabla^k a|=O(r^{-3-k})$ for $k \geq 0$.
    Here $r\co\R^4\to[0,\infty)$  denotes the radius function $r(x):=|x|$.
  \item
    If $(\xi,\omega)\in\ker\delta_I^*$ decays to zero at infinity, then $(\xi,\omega)=0$.
  \end{enumerate}
\end{prop}

In particular, this implies (once more) that $M$ is a smooth manifold and that it can be equipped with an $L^2$ metric arising from the standard metric on $\R^4$.
Clearly, $\Lambda^+:=\Lambda^+(\R^4)^*\iso\so(4)$ acts $\SO(4)$--equivariantly on $\R^4$ and on $\R\oplus\Lambda^+$.
It is a straight-forward computation to verify that the corresponding actions of $\Lambda^+$ on $\Omega^1(\R^4,\fg_E)$ and on $\Omega^0(\R^4,\fg_E)\oplus\Omega^+(\R^4,\fg_E)$ commute with $\delta_I$.
Hence, we obtain an $\SO(4)$--equivariant action of $\Lambda^+$ on $TM$.

\begin{remark}
  If we fix an identification $\R^4=\H$ and correspondingly $\Lambda^+=\Im\H$, then the above defines a hyperkähler structure on $TM$.
  However, for our purpose it is more natural not to fix such an identification.
\end{remark}

$M$ has carries an action of $\R^4\rtimes \R^+$ where $\R^4$ acts by translation and $\R^+$ acts by dilation, i.e., by pullback via $s_\lambda\co \R^4 \to \R^4$ where
\begin{equation*}
  s_\lambda(x):=\lambda x
\end{equation*}
for $\lambda\in\R^+$.
Since the centre of mass of the measure $|F_I|^2\vol$ is equivariant with respect to the $\R^4$--action, we can write
\begin{equation*}
  M=\mathring M \times \R^4
\end{equation*}
where $\mathring M$ is the space of instantons centred at zero.
The action of $\Lambda^+$ preserves this product structure and $\Lambda^+$ acts on the factor $\R^4$ in the usual way.

\begin{example}\label{Ex_ChargeOne}
  If $E$ is the unique $\SU(2)$--bundle over $S^4$ with $c_2(E)=1$, then $E$ carries a single ASD instanton $I$, commonly called ``the one-instanton'', unique up to scaling, translation and changing the framing at infinity.
  We can naturally write the corresponding moduli space as $M=\mathring M\times \R^4=(\slS^+\setminus\{0\})/\Z_2 \times \R^4$.
  Here $S^+$ is the positive spin representation associated with $\R^4$.
\end{example}

\begin{example}
  In general, if $E$ is an $\SU(r)$--bundle over $S^4$, then $M$ can be understood rather explicitly in terms the ADHM construction \cite[Section~3.3]{Donaldson1990}.
\end{example}

\begin{prop}
  \label{Prop_UniversalConnection}
  There exists a $G$--bundle $\bE$ over $M \times S^4$ together with a framing $\bE|_{M\times\set\infty} \to G$ and a tautological connection $\bA \in \sA(\bE)$ on $\bE$ such that:
  \begin{itemize}
  \item
    $\bE|_{\set{[I]}\times S^4} \iso E$ and
  \item
    $\bA$ restricted to ${\set{[I]}\times \R^4}$ is equivalent to $[I]$ via $\sG_0(E)$.
  \end{itemize}
  If we decompose the curvature of the tautological connection $\bA$ over $M\times \R^4$ according to the bi-grading on $\Lambda^* T^*(M\times\R^4)$ induced by $T(M\times\R^4) = \pi_1^*TM \oplus \pi_2^*T\R^4$, then its components satisfy the following:
  \begin{itemize}
  \item
    $F_{\bA}^{2,0} = -2\Delta_I^{-1}\<[a,b]\>$.
  \item
    $F_{\bA}^{1,1} \in \Gamma(\Hom(\pi_1^*TM,\pi_2^*T\R^4\otimes \fg_\bE))$ at $([I],x)$ is the evaluation of $a \in T_{[I]}M = \ker \delta_I$ at $x$;
    in particular, it is $(\R\oplus\Lambda^+)$--linear.
  \item
    $F_{\bA}^{0,2} \in \Gamma(\pi_2^*\Lambda^-(\R^4)^* \otimes \fg_\bE)$.
  \end{itemize}
\end{prop}

\begin{proof}[Proof sketch]
  There is a tautological connection on the pullback of $E$ to $\sA(E) \times S^4$.
  It is flat in the $\sA(E)$--direction.
  It is $\sG_0$--equivariant, but not basic;
  hence, induces a connection on $M\times S^4$ after choosing a connection on $\sA(E) \to \sA(E)/\sG_0(E)$.
  We chose the connection given whose horizontal distribution is given by the Coulomb gauge with respect to the metric on $\R^4$;
  that is, the connection with connection $1$--form $\theta(a) = \Delta_I^{-1}\rd_I^* a$ for $a \in T_I\sA = \Omega^1(\R^n,\fg_E)$.
  The $(2,0)$--component of the curvature of $\bA$ arises from the curvature of this connection.
  The second two bullets are tautological.
\end{proof}


\section{Fueter sections of instanton moduli bundles}
\label{Sec_Fueter}

Let $(Y,\phi)$ be a $\Gtwo$--manifold and let $P$ be an associative submanifold in $Y$.
Fix a moduli space $M$ of framed finite energy ASD instantons on $\R^4$, as in \autoref{Sec_ASDModuli}, and let $E_\infty$ be a $G$--bundle over $P$ together with a connection $A_\infty$.
In the context of \autoref{Thm_A} we take $E_\infty:=E_0|_P$ and $A_\infty:=B|_P$.

\begin{definition}
  The \emph{instanton moduli bundle} $\fM$ over $P$ associated with $E_\infty$ and $M$ is defined by
  \begin{equation*}
    \fM:= ({\rm Fr}(NP)\times E_\infty)\times_{\SO(4)\times G} M.
  \end{equation*}
  Similarly, we define $\mathring\fM$ with $\mathring M$ instead of $M$.
\end{definition}

\begin{example}\label{Ex_MChargeOne}
  Let $M=(\slS^+\setminus\{0\})/\Z_2\times\R^4$ be the moduli space of framed ASD instantons from \autoref{Ex_ChargeOne}.
  If we pick $\fs$ and $\fu$ as in \autoref{Rmk_FueterDirac}, then
  \begin{equation*}
    \fM = (\fs\times\fu\times E_\infty)\times_{\Spin(4)\times\SU(2)} M
        = (\Re(\slS\otimes E_{\infty})\setminus\{0\})/\Z_2 \times NP.
  \end{equation*}
  Here we used the fact that the $\SO(4)$ action on $M$ lifts to an action of $\Spin(4)$.
\end{example}

Denote by $N_\infty P := \Fr(NP)\times_{\SO(4)} S^4$ the sphere-bundle obtained from $NP$ by adjoining a section at infinity.

\begin{theorem}[Donaldson--Segal \cite{Donaldson2009} and Haydys
  \cite{Haydys2011}]
  \label{Thm_DSHaydys}
  To each section $\fI\in\Gamma(\fM)$ we can assign a $G$--bundle $E=E(\fI)$ over $N_\infty P$ together with a connection $I=I(\fI)$ and a framing $\Phi\co E|_{\infty}\to E_\infty$ such that:
  \begin{itemize}
  \item
    For each $x\in P$ the restriction of $I$ to $N_xP$ represents $\fI(x)$.
  \item
    The framing $\Phi$ identifies the restriction of $I$ to the section at infinity with $A_\infty$.
  \end{itemize}
\end{theorem}

The idea of the proof in \cite{Donaldson2009} is to use \autoref{Prop_UniversalConnection} to construct a universal bundle and connection on $\fM \times_{P} N_\infty P$ and to pull those back via $\fI$.

The actions of $\R^+$ on $\R^4$ and $M$ lift to fibre-wise actions on $NP$ and $\fM$.
The construction in \autoref{Thm_DSHaydys} is equivariant with respect to this action.
In particular, $I(s_\lambda^*\fI)=s_\lambda^*I(\fI)$.
It will be convenient to use the shorthand notations
\begin{equation*}
  I_\lambda := I(s_{1/\lambda}^*\fI) \qandq
  \fI_\lambda := s_{1/\lambda}^*\fI.
\end{equation*}

If a section $\fI\in\Gamma(\fM)$ does arise from a sequence of $\Gtwo$--instantons bubbling along $P$, then it is reasonable to expect that in the limit as $\lambda \to 0$ the connection $I_\lambda$ is ``close to being a $\Gtwo$--instanton''.
To make sense of that notion we define the $4$--form $\psi_0$ on $NP$ to be the zeroth order Taylor expansion of $\psi:= \Theta(\phi)$ off $P$.
More explicitly, we can write $\psi_0$ as
\begin{equation}
  \label{Eq_psi0}
  \psi := \vol_{NP}
  - e^1 \wedge e^2 \wedge \omega_{e_3}
  - e^2 \wedge e^3 \wedge \omega_{e_1}
  - e^3 \wedge e^1 \wedge \omega_{e_2}.
\end{equation}
Here $(e_i)$ is a local positive orthonormal frame on $P$, $(e^i)$ is its dual frame, $\vol_{NP}$ is the fibre-wise volume form on $NP$ and $v \in TP \mapsto \omega_v \in \Lambda^+N^*P$ is given by the identification~\eqref{Eq_Clifford}.
With this notation set up the natural requirement is that
\begin{equation}
  \label{Eq_Limiting}
  \lim_{\lambda\to 0}
  \lambda^{-2}s_\lambda^*(F_{I_\lambda}\wedge\psi_0) = F_I \wedge
  (\psi_0-\vol_{NP}) = 0.
\end{equation}
If we introduce a bi-grading on $k$--forms on $NP$ according to the splitting $TNP=\pi_1^*TP\oplus\pi_2^*NP$ corresponding to the connection on $NP$ with $\pi_1\co TP\to P$ and $\pi_2\co NP\to P$ denoting the canonical projections, then it is easy to see that equation~\eqref{Eq_Limiting} splits into two parts.
The first one is simply the condition that the anti-self-dual part of $F_I^{0,2}$ must vanish, while the second part is given by
\begin{align*}
  F_I^{1,1} \wedge\psi_0 = 0.
\end{align*}
This condition can be understood as a partial differential equation on $\fI$ as follows.
Define the vertical tangent bundle $V\fM$ to $\fM$ by
\begin{align*}
  V\fM:= ({\rm Fr}(NP)\times E_\infty)\times_{\SO(4)\times G} TM.
\end{align*}
If $\fI$ is a section of $\fM$, then the action of $\Lambda^+$ on $M$ induces a Clifford multiplication $\gamma\co TP \to \End(\fI^*V\fM)$ in view of the identification~\eqref{Eq_Clifford}.
Moreover, the connections on $NP$ and $E_\infty$ induce a connection $\nabla$ on $\fM$ assigning to each section $\fI$ its covariant derivative $\nabla\fI\in\Omega^1(\fI^*V\fM)$.

\begin{definition}
  The \emph{Fueter operator} $\fF$ associated with $\fM$ is defined by
  \begin{equation*}
    \fI\in\Gamma(\fM) \mapsto \fF\fI:=\sum_{i=1}^3 \gamma(e_i)\nabla_i \fI \in\Gamma(\fI^*V\fM)
  \end{equation*}
  with $(e_i)$ a local orthonormal frame on $P$.
  A section $\fI\in\Gamma(\fM)$ is called a \emph{Fueter section} if it satisfies $\fF\fI=0$.
\end{definition}

\begin{example}
  \label{Eq_FueterChargeOne}
  If $M$ is as in \autoref{Ex_ChargeOne}, then the Fueter operator $\fF$ lifts to the twisted Dirac operator $\slD\co\Gamma(\Re(S\otimes_\C (E_{\infty} \oplus U))\to\Gamma(\Re(S\otimes_\C (E_{\infty} \oplus U))$, cf.~\autoref{Rmk_FueterDirac}.
\end{example}

The Fueter operator $\fF$ is compatible with the product structure on
\begin{equation*}
  \fM=\mathring\fM \times NP
\end{equation*}
 corresponding to $M=\mathring{M}\times\R^4$.
Its restriction to the second factor is given by the Fueter operator $F_P$ associated with $P$.

\begin{theorem}[Donaldson--Segal~\cite{Donaldson2009} and Haydys
  \cite{Haydys2011}]
  \label{Thm_DSHaydysEq}
  If $\fI\in\Gamma(\fM)$, then we can identify $\Gamma(\fI^*V\fM)$ with a subspace of $\Omega^1\(NP, \fg_{E(\fI)}\)$.
  With respect to this identification we have the identity
  \begin{align*}
    \fF\fI=*_0(F_{I(\fI)}^{1,1}\wedge\psi_0)
  \end{align*}
  where $*_0$ is the Hodge--$*$--operator on $NP$.
  In particular, $I(\fI)$ satisfies equation \eqref{Eq_Limiting} if and only if $\fI$ is a Fueter section.
\end{theorem}

\begin{definition}
  The \emph{linearised Fueter operator} $F_\fI=F_{\fI,\phi}\co\Gamma(\fI^*V\fM)\to\Gamma(\fI^*V\fM)$ for $\fI\in\Gamma(\fM)$ is defined by
  \begin{equation}
    \label{Eq_FIphi}
    F_{\fI,\phi}(\hat \fI) := \sum_i \gamma(e_i)\nabla_i \hat \fI 
  \end{equation}
  with $(e_i)$ a local orthonormal frame on $P$.
\end{definition}

\begin{example}
  \label{Ex_LinearisedFueterChargeOne}
  If $M$ is as in \autoref{Ex_ChargeOne}, then the linearised Fueter operator $F_\fI$ lifts to the twisted Dirac operator $\slD\co\Gamma(\Re(\slS\otimes_\C (E_{\infty} \oplus U))\to\Gamma(\Re(\slS\otimes_\C (E_{\infty} \oplus U))$.
  In particular, it only depends on the spin structure $\fs$ and not on $\fI$.
\end{example}

The operator $F_\fI$ is self-adjoint and elliptic;
however, it can never be invertible if $\fI$ is a Fueter section $\fI$.
This is because Fueter sections come in $1$--parameter families $(\fI_\lambda)_{\lambda\in\R+}$.
In particular, taking the derivative at $\lambda=1$ yields an element in the kernel of $F_\fI$.
If $\hat v \in \Gamma(V\fM)$ denotes the vector field generating the action of $\R^+$ on $\fM$, then we can succinctly write this element of the kernel as $\hat v \circ \fI$.

Let $(\phi_t)_{t\in(-T,T)}$ be a family of torsion-free $\Gtwo$--structures on $Y$, let $B_0$ be an unobstructed $\Gtwo$--instanton on a $G$--bundle $E$ over $(Y,\phi_0)$ and let $P_0$ be an unobstructed associative submanifold in $(Y,\phi_0)$.
Then by \autoref{Prop_G2InstantonDeformation} and \autoref{Prop_AssociativeDeformation} we obtain a family of $\Gtwo$--instantons $(B_t)_{t\in(-T',T')}$ over $(Y,\phi_t)$ and a family of associative submanifolds $(P_t)_{t\in(-T',T')}$ in $(Y,\phi_t)$ for some $T'\in(0,T]$.
Now, carry out the above construction with $P=P_t$, $E_\infty=E|_{P_t}$, $A_\infty=B_t|_{P_t}$ and a fixed moduli space $M$ of framed finite energy ASD instantons to obtain a family of instanton moduli bundles $(\fM_t)_{t\in(-T',T')}$ along with a family of Fueter operators $(\fF_t)_{t\in(-T',T')}$.
If $\fI_0$ is a Fueter section of $\fM_0$ with
\begin{equation*}
  \dim\ker F_{\fI_0}=1,
\end{equation*}
then, using the implicit function theorem, we obtain a family $(\fI_t)_{t\in(-T',T')}$ of sections of $\fM_t$ satisfying
\begin{align}\label{Eq_SpectralFlow}
  \fF_t\fI_t+\mu(t)\cdot\hat v\circ\fI_t=0
\end{align}
where $\mu\co (-T',T') \to \R$ is a smooth function vanishing at zero.

\begin{definition}\label{Def_FueterUnobstructed}
  In the above situation we say that $\fI_0$ is \emph{unobstructed} with respect to $(\phi_t)$ if
  \begin{align*}
    \left.\frac{\del\mu}{\del t}\right|_{t=0} \neq 0.
  \end{align*}
\end{definition}

\begin{remark}
  \label{Rmk_WeaklyUnobstructed}
  One can work with a slightly weaker notion of unobstructedness where one only requires that $\mu$ is strictly monotone near $t = 0$.
  A slight variation of \autoref{Thm_A} still holds in this case.
  We will pick up this thread again in \autoref{Sec_Conclusion}.
\end{remark}

\begin{example}
  If $M$ is as in \autoref{Ex_ChargeOne}, then equation~\eqref{Eq_SpectralFlow} can be viewed as the spectral flow of a family of twisted Dirac operators and $\fI_0$ is unobstructed if and only if this spectral flow has a regular crossing at $\fI_0$.
\end{example}


\section{Pregluing construction}
\label{Sec_Pregluing}

In this section we begin the proof of \autoref{Thm_A} in earnest.
Suppose that $Y$, $(\phi_t)_{t\in(-T,T)}$, $B$, $P$ and $\fI \in \Gamma(\fM)$ are as in the hypothesis of \autoref{Thm_A}.

\begin{convention}
  We fix constants $T'\in(0,T]$ and $\Lambda>0$ such that all of the statements of the kind ``if $t\in(-T',T')$ and $\lambda\in(0,\Lambda]$, then \ldots'' appearing in the following are valid.
  This is possible since there is only a finite number of these statements and each one of them is valid provided $T'$ and $\Lambda$ are sufficiently small.
  By $c>0$ we will denote a generic constant whose value depends neither on $t\in(-T',T')$ nor on $\lambda\in(0,\Lambda]$ but may change from one occurrence to the next.
\end{convention}

As discussed at the end of \autoref{Sec_Fueter}, $B_0 := B$ and $P_0 := P$ give rise to:
\begin{itemize}
\item 
  a family $(B_t)_{t\in(-T',T')}$ of $\Gtwo$--instantons on $E_0$ over $(Y,\phi_t)$,
\item 
  a family of associative submanifolds $(P_t)_{t\in(-T',T')}$ in $(Y,\phi_t)$ and, hence,
\item
  a family of instanton moduli bundles $(\fM_t)_{t\in(-T',T')}$ with $\fM_0=\fM$ and Fueter operators $(\fF_t)_{t\in(-T',T')}$ together with sections $(\fI_t)_{t\in(-T',T')}$ satisfying $\fI_0=\fI$ and
  \begin{equation}
    \label{Eq_SpectralFlow2}
    \fF_t \fI_t + \mu(t) \hat v \circ \fI_t = 0
  \end{equation}
  where $\mu\co (-T',T') \to \R$ is a smooth function vanishing at zero with
  \begin{equation*}
    \left.\frac{\del\mu}{\del t}\right|_{t=0} \neq 0.
  \end{equation*}
\end{itemize}
The proof of \autoref{Thm_A} proceeds via a gluing construction.
As a first step we explain how to construct approximate solutions.

\begin{prop}
  \label{Prop_Pregluing}
  For each $t\in(-T',T')$ and $\lambda\in(0,\Lambda]$ we can explicitly construct a $G$--bundle $E_{t,\lambda}$ together with a connection $A_{t,\lambda}=A_t\#_\lambda\fI_t$ from $E_0$, $A_t\in\sA(E_0)$ and $\fI_t$.
  The bundles $E_{t,\lambda}$ are pairwise isomorphic.
\end{prop}

Before we embark on the proof, let us set up some notation.
Fix a constant $\sigma>0$ such that for all $t\in(-T',T')$ the exponential map identifies a tubular neighbourhood of width $8\sigma$ of $P_t$ in $Y$ with a neighbourhood of the zero section in $NP_t$.
For $I\subset\R$ we set
\begin{equation*}
  U_{I,t} := \{ v\in NP_t : |v|\in I \} \qandq
  V_{I,t} := \{ x \in Y : r_t(x)\in I \}.
\end{equation*}
Here $r_t:=d(\cdot,P_t)\co Y\to [0,\infty)$ denotes the distance from $P_t$.
Fix a smooth-cut off function $\chi\co[0,\infty)\to[0,1]$ which vanishes on $[0,1]$ and is equal to one on $[2,\infty)$.
For $t\in(-T',T')$ and $\lambda\in(0,\Lambda]$ we define $\chi^-_{t,\lambda}\co Y\to[0,1]$ and $\chi^+_t\co Y\to[0,1]$ by
\begin{equation*}
  \chi^-_{t,\lambda}(x) := \chi(r_t(x)/2\lambda) \qandq
  \chi^+_t(x) := 1-\chi(r_t(x)/2\sigma),
\end{equation*}
respectively.

\begin{proof}[Proof of Proposition~\ref{Prop_Pregluing}]
  Via radial parallel transport we can identify $E(\fI_t)$ over $U_{(R,\infty),t}$ for some $R>0$ with the pullback of $E(\fI_t)|_{\infty}$ to said region and similarly we can identify $E_0$ over $V_{[0,\sigma),t}$ with the pullback of $E_0|_{P_t}$.
  Hence, via the framing $\Phi$ we can identify $s_{1/\lambda}^*E(\fI_t)$ with $E_0$ on the overlap $V_{(\lambda,\sigma),t}$ for $\lambda \in (0,\Lambda]$. 
  Patching both bundles via this identification yields $E_{t,\lambda}$.

  To construct a connection on $E_{t,\lambda}$ note that on the overlap $I_{t,\lambda}:=s_{1/\lambda}^*I(\fI_t)$ and $B_t$ can be written as
  \begin{equation*}
    I_{t,\lambda}=B_t|_{P_t}+i_{t,\lambda} \qandq
    B_t=B_t|_{P_t}+b_t.
  \end{equation*}
  Here and in the following, by a slight abuse of notation, we denote by $B_t|_{P_t}$ the pullback of $B_t|_{P_t}$ to the overlap.
  We define $A_{t,\lambda}$ by interpolating between $I_\lambda$ and $B_t$ on the overlap as follows
  \begin{equation}
    \label{Eq_PregluedConnection}
    A_{t,\lambda}:=B_t|_{P_t}+\chi^-_{t,\lambda}b_t+\chi^+_ti_{t,\lambda}.
    \qedhere
  \end{equation}
\end{proof}

Now, in view of Proposition~\ref{Prop_G2Instanton}, the task at hand is to solve the equation
\begin{equation}
  \label{Eq_Main}
  *_{\phi_t}\(F_{A_{t,\lambda}+a}\wedge\psi_t\)+\rd_{A_{t,\lambda}+a}\xi=0
\end{equation}
where $\psi_t:=\Theta(\phi_t)=*_{\phi_t}\phi_t$, $t=t(\lambda)$, $a=a(\lambda)$ and $\xi=\xi(\lambda)$.
If we could find an appropriate analytic setup in which $(a,\xi)=0$ becomes closer and closer to being a solution of equation~\eqref{Eq_Main} while at the same time the linearisations $L_{t,\lambda}:=L_{A_{t,\lambda},\phi_t}$, as defined in \eqref{Eq_LAphi}, possess right inverses that can be controlled \emph{uniformly} in $t$ and $\lambda$, then it would not be too difficult to solve equation~\eqref{Eq_Main} for all $t\in(-T',T')$ and $\lambda\in(0,\Lambda]$.
Since the properties of $L_{t,\lambda}$ are closely linked, among other things, to those of $F_{\fI_t}$ and since $F_{\fI_0}$ has a one-dimensional cokernel, however, we will only be able to solve equation~\eqref{Eq_Main} ``modulo the cokernel of $F_{\fI_0}$''.
More precisely, when $t\in(-T',T')$ and $\lambda\in(0,\Lambda]$ we will be able to solve the equation
\begin{equation}
  \label{Eq_MainExtended}
  L_{t,\lambda}\ua + \eta \cdot \iota_{t,\lambda}\hat v\circ \fI_t +
  Q_{t,\lambda}(\ua)
  + e_{t,\lambda}=0
\end{equation}
for $\ua=(\xi,a)\in\Omega^0(Y,\fg_{E_{t,\lambda}}) \oplus\Omega^1(Y,\fg_{E_{t,\lambda}})$ and $\eta\in\R$ with $Q_{t,\lambda}$ and $e_{t,\lambda}$ defined by
\begin{equation}
  \label{Eq_QuadraticTerm}
  Q_{t,\lambda}(\ua):=
 \frac12*\(\left[a\wedge
 a\right]\wedge\psi_t\)
 +[\xi,a].
\end{equation}
and
\begin{equation*}
  e_{t,\lambda}
  := *(F_{A_{t,\lambda}}\wedge\psi_t)
  + \mu(t)\cdot\iota_{t,\lambda}\hat v\circ \fI_t,
\end{equation*}
respectively.
Here the map $\iota_{t,\lambda}\co\Gamma(\fI_t^*V\fM_t)\to\Omega^1(Y,\fg_{E_{t,\lambda}})$ is defined by
\begin{equation*}
  \iota_{t,\lambda}\hat\fI:=\chi^+_ts_{1/\lambda}^*\hat\fI
\end{equation*}
where we first identify $\hat\fI\in\Gamma(\fI_t^*V\fM_t)$ with an element of $\Omega^1\(NP,E(\fI_t)\)$, then view the restriction of its
pullback via $s_\lambda^{-1}$ to $U_{[0,\sigma),t}$ as lying in $\Omega^1(V_{[0,\sigma),t},\fg_{E_{t,\lambda}})$ and finally extended it to all of $Y$ by multiplication with $\chi^+_t$.
After solving \eqref{Eq_MainExtended} we are left with the residual scalar equation
\begin{equation*}
  \mu(t) + \eta(t,\lambda) = 0.
\end{equation*}
It will turn out that $\eta$ and $\del_t\eta$ go to zero as $\lambda \to 0$.
Since $\del_t \mu(0) \neq 0$, finding $t=t(\lambda)$ such that equation~\eqref{Eq_Main} is satisfied is then a simple consequence of an implicit function theorem.

Let us now discuss some aspects of the analysis.
First of all we will introduce appropriate weighted Hölder spaces in \autoref{Sec_Holder}.
One should think of these weighted spaces as a convenient framework to deal with different local scales simultaneously.
In our case they are constructed to counteract the fact that the curvature of the connection $A_{t,\lambda}$ around $P_t$ becomes larger and larger as $\lambda \to 0$.
We will see in \autoref{Sec_PregluingEstimate} that the amount by which our approximate solutions $A_{t,\lambda}$ fail to be solutions of equation~\eqref{Eq_Main} ``modulo the cokernel of $F_{\fI_0}$'' measured in our weighted Hölder norms goes to zero at a certain rate as $\lambda \to 0$.
The key difficulty then lies in analysing the linearisation $L_{t,\lambda}$.
As is the case in most adiabatic limit constructions, the linearisation $L_{t,\lambda}$ is rather badly behaved on an infinite dimensional space:
For every $\hat\fI\in\Gamma(\fI_t^*V\fM_t)$ the appropriate norm of $\iota_{t,\lambda} \hat\fI$ is essentially independent of $\lambda$, while the appropriate norm of $L_{t,\lambda} \iota_{t,\lambda} \hat\fI$ tends to zero as $\lambda \to 0$.
To overcome this issue it is convenient to split the problem at hand into a part coming from $\Gamma(\fI_t^*V\fM_t)$ and the part orthogonal to it.
We define $\pi_{t,\lambda}\co\Omega^1(Y,\fg_{E_{t,\lambda}}) \to
\Gamma(\fI_t^*V\fM_t)$ by
\begin{equation*}
  (\pi_{t,\lambda} a)(x)
  :=\sum_\kappa \int_{N_xP} \<a,\iota_{t,\lambda}\kappa\> \kappa
\end{equation*}
for $x\in P_t$. 
Here $\kappa$ runs through an orthonormal basis of $(V\fM_t)_{\fI(x)}$ with respect to the inner product $\<\iota_{t,\lambda}\cdot,\iota_{t,\lambda}\cdot\>$.
Clearly, $\pi_{t,\lambda} \iota_{t,\lambda}=\id$;
hence, $\bar\pi_{t,\lambda}:=\iota_{t,\lambda}\pi_{t,\lambda}$ is a projection.
We denote the complementary projection by $\rho_{t,\lambda} := \id-\bar\pi_{t,\lambda}$.
If we define
\begin{equation*}
  \fA_{t,\lambda}:=\Omega^0(Y,\fg_{E_{t,\lambda}}) \oplus \ker \pi_{t,\lambda},
\end{equation*}
then we can write
\begin{equation*}
  \Omega^0(Y,\fg_{E_{t,\lambda}})\oplus\Omega^1(Y,\fg_{E_{t,\lambda}}) = \fA_{t,\lambda} \oplus \Gamma(\fI_{t}^*V\fM_t)
\end{equation*}
and decompose $L_{t,\lambda}$ accordingly into a $2$--by--$2$ matrix of operators.
We will see in \autoref{Sec_LinearAnalysis} that the diagonal entries can be controlled in terms of certain models on $\R^7$, $L_{A_0}$ and the linearised Fueter operator $F_{\fI_0}$, while the off-diagonal terms are negligibly small.
In \autoref{Sec_QuadraticEstimate} we discuss how to control the non-linearity $Q_{t,\lambda}$ in equation~\eqref{Eq_MainExtended}.
The completion of the proof of \autoref{Thm_A} in \autoref{Sec_Conclusion} will then be rather straight-forward.


\section{Weighted Hölder norms}
\label{Sec_Holder}

For $t\in(-T',T')$ and $\lambda\in(0,\Lambda]$ we define a family of weight functions $w_{\ell,\delta;t,\lambda}$ on $Y$ depending on two additional parameters $\ell,\delta\in\R$ as follows
\begin{align*}
  w_{\ell,\delta;t,\lambda}(x) :=
  \begin{cases}
    \lambda^{\delta}(\lambda+r_t(x))^{-\ell-\delta}
    &\text{if}~r_t(x)\leq \sqrt\lambda \\
    r_t(x)^{-\ell+\delta} &\text{if}~r_t(x) > \sqrt\lambda
  \end{cases}
\end{align*}
and set $w_{\ell,\delta;t,\lambda}(x,y) :=
\min\{w_{\ell,\delta;t,\lambda}(x), w_{\ell,\delta;t,\lambda}(x)\}.$
For a Hölder exponent $\alpha\in(0,1)$ and $\ell,\delta\in\R$ we
define (semi-)norms
\begin{align*}
  \|f\|_{L^\infty_{\ell,\delta;t,\lambda}(U)}
    &:=\|w_{\ell,\delta;t,\lambda}f\|_{L^\infty(U)}, \\
  [f]_{C^{0,\alpha}_{\ell,\delta;t,\lambda}(U)}
    &:=\sup_{\stackrel{x\neq y\in U:}{d(x,y)\leq \lambda + \min\{r_t(x),r_t(y)\}}}
      w_{\ell-\alpha,\delta;t,\lambda}(x,y) \frac{|f(x)-f(y)|}{d(x,y)^\alpha} \qand \\
  \|f\|_{C^{k,\alpha}_{\ell,\delta;t,\lambda}(U)} 
    &:=\sum_{j=0}^k
    \|\nabla^{k} f\|_{L^\infty_{\ell-j,\delta;t,\lambda}(U)}
    + [\nabla^{k}f]_{C^{0,\alpha}_{\ell-j,\delta;t,\lambda}}.
\end{align*}
Here $f$ is a section of a vector bundle over $U\subset Y$ equipped
with an inner product and a compatible connection.  We use parallel
transport to compare the values of $f$ at different points.  If $U$ is
not specified, then we take $U=Y$.  We will primarily use this norm
for $\fg_{E_{t,\lambda}}$--valued tensor fields.

\begin{remark}
  The reader may find the following heuristic useful.  Let $f$ be a
  $k$--form on $Y$.  Fix a small ball centred at a point $x\in P_t$,
  identify it with a small ball in $T_xY=T_xP_t\oplus N_xP_t$ and
  rescale this ball by a factor $1/\lambda$.  Upon pulling everything
  back to this rescaled ball the weight function
  $w_{-k,\delta,t,\lambda}$ becomes essentially
  $\lambda^k(1+|y|)^{k-\delta}$, where $y$ denotes the
  $N_xP_t$--coordinate.  Thus as $\lambda$ goes to zero a uniform
  bound $\|f_\lambda\|_{L^\infty_{-k,\delta,t,\lambda}}$ on a family
  $(f_\lambda)$ of $k$--forms ensures that the pullbacks of
  $f_\lambda$ decay like $|y|^{-k+\delta}$ in the direction of
  $N_xP_t$.  At the same time it forces $f_\lambda$ not to blowup at a
  rate faster than $r_t^{-k-\delta}$ along $P_t$.  The ``discrepancy''
  in the exponents can be seen to be rather natural by considering the
  action of the inversion $y \mapsto \lambda y/|y|^2$.
\end{remark}

\begin{prop}
  \label{Prop_Multiplication}
  If $(f,g)\mapsto f\cdot g$ is a bilinear form satisfying
  $|f\cdot g|\leq |f||g|$, then
  \begin{equation*}
    \|f\cdot
    g\|_{C^{k,\alpha}_{\ell_1+\ell_2,\delta_1+\delta_2;t,\lambda}}
    \leq \|f\|_{C^{k,\alpha}_{\ell_1,\delta_1;t,\lambda}}
    \|g\|_{C^{k,\alpha}_{\ell_2,\delta_2;t,\lambda}}.
  \end{equation*}
\end{prop}

\begin{proof}
  This follows immediately from the above definition.
\end{proof}

\begin{cor}
  \label{Cor_DeltaSensitivity}
  If $\delta<0$, then there is a constant $c>0$ which is independent
  of $t\in(-T',T')$ and $\lambda\in(0,\Lambda]$ such that
  \begin{equation*}
    \|f\|_{C^{k,\alpha}_{\ell,\delta;t,\lambda}}
    \leq c\lambda^{\delta/2} \|f\|_{C^{k,\alpha}_{\ell,0;t,\lambda}}
    \qandq
    \|f\|_{C^{k,\alpha}_{\ell,0;t,\lambda}}
    \leq c \|f\|_{C^{k,\alpha}_{\ell,\delta;t,\lambda}}
  \end{equation*}
\end{cor}

\begin{proof}
  Use $\|1\|_{C^{k,\alpha}_{0,\delta;t,\lambda}} \leq c\lambda^{\delta/2}$ and
  $\|1\|_{C^{k,\alpha}_{0,-\delta;t,\lambda}} \leq c$ for $\delta<0$.
\end{proof}

\begin{prop}
  \label{Prop_IotaPi}
  For $\ell\leq-1$ and $\delta\in\R$ such that $\ell-\alpha+\delta > -3$ and $\ell + \delta < -1$ there is a constant $c>0$ such that for all
  $t\in(-T',T')$ and $\lambda\in(0,\Lambda]$
  we have
  \begin{align*}
    \|\iota_{t,\lambda}\hat\fI\|_{C^{0,\alpha}_{\ell,\delta;t,\lambda}}
      &\leq c\lambda^{-1-\ell} \|\hat\fI\|_{C^{0,\alpha}} \qandq \\
    \|\pi_{t,\lambda} a\|_{C^{0,\alpha}}
      &\leq c
      \lambda^{1+\ell-\alpha}\|a\|_{C^{0,\alpha}_{\ell,\delta;t,\lambda}(V_{[0,\sigma);t})}.
  \end{align*}
  In particular,
  $\bar\pi_{t,\lambda}=\iota_{t,\lambda}\pi_{t,\lambda}$ and
  $\rho_{t,\lambda}$ are bounded by $c\lambda^{-\alpha}$ with respect
  to the $C^{0,\alpha}_{\ell,\delta;t,\lambda}$--norms.
\end{prop}

\begin{proof}
  From \autoref{Prop_Decay} it follows at once that
  \begin{equation*}
    \|s_{1/\lambda}^*\hat\fI\|_{C^{0,\alpha}_{-3,0;t,\lambda}\(V_{[0,\sigma),t}\)}
    \leq c \lambda^2 \|\hat\fI\|_{C^{0,\alpha}}.
  \end{equation*}
  The first inequality thus is a consequence of
  \autoref{Prop_Multiplication} since
  $\|\chi^+_t\|_{C^{0,\alpha}_{3+\ell,\delta;t,\lambda}} \leq c
  \lambda^{-3-\ell}$ for $\ell+\delta>-3$.  

  To prove the second inequality, note that by
  \autoref{Prop_Decay} for $\kappa\in(V\fM_t)_{\fI_t(x)}$ we
  have $|s_{1/\lambda}^*\kappa|(x) \leq
  c\lambda^2/(\lambda+|x|)^3\|\kappa\|_{L^2}$ and thus
  \begin{align*}
    \int_{N_xP} \<a,\chi^+_ts_{1/\lambda}^*\kappa\>
    &\leq c \int_0^{\sqrt\lambda}
    \lambda^{2-\delta}(\lambda+r)^{\ell+\delta-3}r^3 \rd r 
        \cdot \|a\|_{L^\infty_{\ell,\delta;t,\lambda}}\|\kappa\|_{L^2}\\
    &\qquad +c \int_{\sqrt\lambda}^\sigma \lambda^2
    r^{\ell-\delta}(\lambda+r)^{-3} r^3 \rd r
     \cdot \|a\|_{L^\infty_{\ell,\delta;t,\lambda}}\|\kappa\|_{L^2} \\
    &\leq c\lambda^{3+\ell} \|a\|_{L^\infty_{\ell,\delta;t,\lambda}}\|\kappa\|_{L^2}
  \end{align*}
  since $\ell\leq-1$ and $\ell+\delta<-1$.  If $\kappa$ is an element
  of an orthonormal basis of $(V\fM_t)_{\fI_t(x)}$ with respect to
  $\<\iota_{t,\lambda}\cdot,\iota_{t,\lambda}\cdot\>$, then
  $\|\kappa\|_{L^2}\leq c/\lambda$ since for $\kappa_1,\kappa_2 \in
  (V\fM_t)_{\fI_t(x)}$
  \begin{equation*}
    \lambda^2 \<\kappa_1,\kappa_2\>_{L^2} \sim
    \<\chi^+_ts_{1/\lambda}^*\kappa_1,\chi^+_ts_{1/\lambda}^*\kappa_2\>_{L^2}
  \end{equation*}
  where $\sim$ means comparable uniformly in $t$ and $\lambda$.  Therefore,
  \begin{equation*}
    \|\pi_{t,\lambda} a\|_{L^\infty} \leq c \lambda^{1+\ell}
    \|a\|_{L^\infty_{\ell,\delta;t,\lambda}}.
  \end{equation*}
  The estimates on the Hölder norms follow by the same kind of
  argument.
\end{proof}


\section{Pregluing estimate}
\label{Sec_PregluingEstimate}

In the following we will need to differentiate various tensors over
$Y$ and $NP_t$ depending on $t\in(-T,T)$.  For tensors over $Y$ we
could simply differentiate using $\del_t$; however, $\del_t$ does not
drag $P_t$ along in a parallel fashion, which causes some additional
error terms to show up, and also it is preferable to differentiate
tensors over $Y$ and $NP_t$ in way that is consistent with the
identification $V_{I,t}=U_{I,t}$ for $I\subset [0,2\sigma)$.
Therefore we use a fixed set of connections constructed as follows:
For each $t \in (-T',T')$ we can write $P_t = \{ \exp_p(v_t) : p \in P_0 \}$ for
some unique normal vector field $v_t \in \Gamma(P_0,NP_0)$; hence, the
bundle $\coprod_{t\in(-T',T')} P_t \to (-T',T')$ comes with a canonical
connection.  Pick a connection on $\sN:=\coprod_{t\in(-T',T')} NP_t \to
(-T',T')$ such that for each parallel path $t\mapsto p_t\in P_t$ its
lift to the zero section $t\mapsto 0_{p_t}\in NP_t$ is also parallel.
Moreover, we pick a connection on $Y\times(-T',T') \to (-T',T')$ which
agrees with the connection on $\sN$ on $\bigcup_{t\in(-T',T')}
V_{[0,2\sigma),t}=\bigcup_{t\in(-T',T')} U_{[0,2\sigma),t}$ and with
$\del_t$ on $Y\setminus \bigcup_{t\in(-T',T')} V_{[0,4\sigma),t}$.
These connections induce various connections on bundles of tensors
over $Y$ and $NP_t$; we denote the associated covariant derivatives by
$\nabla_t$.

\begin{prop}
  \label{Prop_PregluingEstimate}
  There is a constant $c>0$ such that for $t\in(-T',T')$ and
  $\lambda\in(0,\Lambda]$ we have
  \begin{equation*}
    \|e_{t,\lambda}\|_{C^{0,\alpha}_{-2,0;t,\lambda}}\leq c \lambda^2
    \qandq
    \|\nabla_t e_{t,\lambda}\|_{C^{0,\alpha}_{-2,0;t,\lambda}}\leq c \lambda^2.
  \end{equation*}
\end{prop}

The proof of this result requires some preparation.

\begin{prop}
  \label{Prop_PsiExpansion}
  In the tubular neighbourhood $V_{[0,\sigma);t}$ of $P_t$ we can
  write $\psi_t:=\Theta(\phi_t)=*_{\phi_t}\phi_t$ as
  \begin{equation*}
    \psi_t=\psi_{0;t} + \psi_{1;t} + \psi_{\geq 2;t}
  \end{equation*}
  where $\psi_{0;t}$ is defined as in equation~\eqref{Eq_psi0},
  $\psi_{1;t}$ takes values in $\Lambda^2 T^*P_t\otimes \Lambda^+
  N^*P_t$.  Moreover, $\psi_{0;t}$, $\psi_{1;t}$ and $\psi_{\geq 2;t}$
  depend continuously differentiably on $t$, and there is a constant
  $c>0$ such that for all $t\in(-T',T')$ we have
  \begin{equation*}
    \|\psi_{0;t}\|_{C^{0,\alpha}_{0,0;t,\lambda}(V_{[0,\sigma);t})}
    +\|\psi_{1;t}\|_{C^{0,\alpha}_{1,0;t,\lambda}(V_{[0,\sigma);t})}
    + \|\psi_{\geq 2;t}\|_{C^{0,\alpha}_{2,0;t,\lambda}(V_{[0,\sigma);t})}
   \leq c
   \end{equation*}
   and
   \begin{equation*}
    \|\nabla_t\psi_{0;t}\|_{C^{0,\alpha}_{0,0;t,\lambda}(V_{[0,\sigma);t})}
    +\|\nabla_t\psi_{1;t}\|_{C^{0,\alpha}_{1,0;t,\lambda}(V_{[0,\sigma);t})}
    +\|\nabla_t\psi_{\geq  2;t}\|_{C^{0,\alpha}_{2,0;t,\lambda}(V_{[0,\sigma);t})}
    \leq c.
  \end{equation*}
\end{prop}

\begin{proof}[Proof of \autoref{Prop_PsiExpansion}]
  Let $P=P_t$ and $\psi=\psi_t$.  If we pull the identity map of a
  tubular neighbourhood of $P$ back to a tubular neighbourhood of the
  zero section of $NP$ via the exponential map, then the Taylor
  expansion of its derivative around $P$ can be expressed in the
  splitting $TNP=\pi_1^*TP\oplus\pi_2^*NP$ as
  \begin{equation*}
    (x,y) \mapsto (x, y) + \(\rII_y(x),y\) + O\(|y|^2\)
  \end{equation*}
  where $\rII$ is the second fundamental form of $P$ in $Y$ which we
  think of as a map from $NP$ to $\End(TP)$.  This immediately yields
  the desired expansion of $\psi$ near $P$, with $\psi_1$ taking
  values in $\Lambda^2 T^*P\otimes \Lambda^+ N^*P$, since we know that
  $\psi$ is given by $\psi_0$ along $P$.  Moreover, we have $\nabla^k
  \psi_1 = O\(|y|^{1-k}\)$ and $\nabla^k \psi_{\geq 2} =
  O\(|y|^{2-k}\)$ for $k=0,1$ which implies the first estimate and,
  since everything depends smoothly on $t$, also the second estimate.
\end{proof}

The same reasoning also proves the following result.

\begin{prop}\label{Prop_*Expansion}
  There is a constant $c>0$ such that for all $t\in(-T',T')$ and
  $\lambda\in(0,\Lambda]$ we have
  \begin{equation*}
    \|*_0-*\|_{C^{0,\alpha}_{1,0;t,\lambda}(V_{[0,\sigma);t})}
    +\|\nabla_t(*_0-*)\|_{C^{0,\alpha}_{1,0;t,\lambda}(V_{[0,\sigma);t})}
   \leq c.
  \end{equation*}
\end{prop}

\begin{prop}\label{Prop_F}
  There is a constant $c>0$ such that for all $t\in(-T',T')$ and
  $\lambda\in(0,\Lambda]$ we have
  \begin{align*}
   \left\|F_{I_{t,\lambda}}^{2,0}-F_{B_t|_{P_t}}\right\|_{C^{0,\alpha}_{-2,0;t,\lambda}(V_{[0,\sigma);t})}+
\left\|\nabla_t\(F_{I_{t,\lambda}}^{2,0}-F_{B_t|_{P_t}}\)\right\|_{C^{0,\alpha}_{-2,0;t,\lambda}(V_{[0,\sigma);t})}
      &\leq c\lambda^2,  \\
         \left\|F_{I_{t,\lambda}}^{1,1}\right\|_{C^{0,\alpha}_{-3,0;t,\lambda}(V_{[0,\sigma);t})}
+   \left\|\nabla_t F_{I_{t,\lambda}}^{1,1}\right\|_{C^{0,\alpha}_{-3,0;t,\lambda}(V_{[0,\sigma);t})}   
   &\leq c \lambda^2 \\ \andq
       \left\|F_{I_{t,\lambda}}^{0,2}\right\|_{C^{0,\alpha}_{-4,0;t,\lambda}(V_{[0,\sigma);t})}+
     \left\|\nabla_tF_{I_{t,\lambda}}^{0,2}\right\|_{C^{0,\alpha}_{-4,0;t,\lambda}(V_{[0,\sigma);t})}      &\leq c \lambda^2.
    \end{align*}
\end{prop}

\begin{proof}
  \autoref{Thm_DSHaydys} asserts that the restriction of $I_t=I(\fI_t)$ to the section at infinity agrees with $B_t|_{P_t}$.
  For a local coordinate system $(z_1,\ldots,z_3,w_1,\ldots,w_4)$ based at a point on the section at infinity and with $z_i$ denoting the coordinates along $P_t$ and $w_i$ denote transverse coordinates we can write
  \begin{equation*}
    I_t = B_t|_{P_t} + \sum_{i,j} w_i(\xi_{ij} \rd z_j + \eta_{ij} \rd w_j) + O(|w|^2)
  \end{equation*}
  for $\xi_{ij},\eta_{ij} \in \fg$.
  It follows that $F_{I_t}^{1,1} = -\sum_{i,j=1}^4 \xi_{ij} \rd z_i \wedge \rd w_j + O(|w|)$.
  However, by \autoref{Prop_UniversalConnection} and the paragraph after \autoref{Thm_DSHaydys}, for any fixed $v \in T_xP_t$, $i(v)F_{I_t}^{1,1} \in \ker \delta_{\fI_t(p)}$;
  hence, by \autoref{Prop_Decay} this curvature component decays like $r^{-3}$ when viewed from the zero section.
  This translates into $\xi_{ij} = 0$, and we can write
  \begin{equation}
    \label{Eq_IFromInfty}
    I_t = B_t|_{P_t} + \sum_{i,j=1}^4 \eta_{ij} w_i\rd w_j + O(|w|^2).
  \end{equation}

  Hence, $F_{I_t}^{2,0}-F_{B_t|_{P_t}}$ vanishes to first order along the section at infinity which when viewed from the zero section in $NP_t$ means that
  \begin{equation*}
    \left|F_{I_t}^{2,0}-F_{B_t|_{P_t}}\right|\leq\frac{c}{1+|y|^2}.
  \end{equation*}
  The first estimate now follows from a simple scaling consideration
  and by realising that the above reasoning also applies to
  $\nabla_t\(F_{I_t}^{2,0}-F_{B_t|_{P_t}}\)$.

  The last two estimates follow from \autoref{Thm_DSHaydys},
  \autoref{Prop_Decay}, the fact that the curvature of a
  finite energy ASD instanton decays at least like $|y|^{-4}$ and
  simple scaling considerations.
\end{proof}

\begin{prop}\label{Prop_iae}
  There is a constant $c>0$ such that for all $t\in(-T',T')$ and
  $\lambda\in(0,\Lambda]$ we have
  \begin{align*}
    \|i_{t,\lambda}\|_{C^{0,\alpha}_{-3,0;t,\lambda}(V_{(\lambda,\sigma);t})}
    +\|\rd_{I_{t,\lambda}}i_{t,\lambda}\|_{C^{0,\alpha}_{-4,0;t,\lambda}(V_{(\lambda,\sigma);t})}
   &\leq c\lambda^2 \qand \\
    \|\nabla_ti_{t,\lambda}\|_{C^{0,\alpha}_{-3,0;t,\lambda}(V_{(\lambda,\sigma);t})}
    +\|\nabla_t(\rd_{I_{t,\lambda}}i_{t,\lambda})\|_{C^{0,\alpha}_{-4,0;t,\lambda}(V_{(\lambda,\sigma);t})}
   &\leq c
  \lambda^2 \intertext{as well as}
    \|b_t\|_{C^{0,\alpha}_{1,0;t,\lambda}(V_{[0,\sigma);t})}
    +\|\rd_{B_t|_{P_t}}
    b_t\|_{C^{0,\alpha}_{0,0;t,\lambda}(V_{[0,\sigma);t})}
   &\leq c \qand \\
    \|\nabla_tb_t\|_{C^{0,\alpha}_{1,0;t,\lambda}(V_{[0,\sigma);t})}
    + \|\nabla_t(\rd_{B_t|_{P_t}}b_t)\|_{C^{0,\alpha}_{0,0;t,\lambda}(V_{[0,\sigma);t})}
   &\leq c.
  \end{align*}
\end{prop}

\begin{proof}
  The first two estimates follow from \eqref{Eq_IFromInfty} and a
  simple scaling consideration, while the last two estimates follow
  from the fact that we put $B_t$ into radial gauge from the zero section
  in $NP_t$.
\end{proof}

\begin{proof}[Proof of \autoref{Prop_PregluingEstimate}]
  We proceed in four steps.
  First we estimate $\tilde e_{t,\lambda}$, an approximation  of $e_{t,\lambda}$.
  Then we estimate the
  difference $e_{t,\lambda}-\tilde e_{t,\lambda}$ separately in the
  three subsets $V_{[0,\lambda);t}$, $V_{[\lambda,\sigma/2);t}$ and
  $V_{[\sigma/2,\sigma);t}$ constituting $V_{[0,\sigma);t}$ which
  contains the support of $e_{t,\lambda}$.

  It will be convenient to use the following shorthand notation
  \begin{equation*}
    \|f\|_{{\ell,U}} := \|f\|_{C^{0,\alpha}_{\ell,0;t,\lambda}(U)} +
    \|\nabla_t f\|_{C^{0,\alpha}_{\ell,0;t,\lambda}(U)}.
  \end{equation*}
  Note that if $(f,g)\mapsto f\cdot g$ is a bilinear map satisfying
  $|f\cdot g|\leq|f||g|$ and the Leibniz rule with respect to
  $\nabla_t$, then it follows from \autoref{Prop_Multiplication} that
  $\|f\cdot g\|_{\ell_1+\ell_2,U}\leq \|f\|_{\ell_1,U}\cdot
  \|g\|_{\ell_2,U}$.

  \setcounter{step}{0}
  \begin{step}\label{Step_x1}
    The term
    \begin{equation*}
      \tilde e_{t,\lambda}:=*\left[\(F_{I_{t,\lambda}} -
        F_{B_t|_{P_t}}\)\wedge\psi_t\right] + \mu(t)\cdot \hat v \circ
      \fI_{t,\lambda}
    \end{equation*}
    satisfies $\|\tilde e_{t,\lambda}\|_{-2,V_{[0,\sigma);t}} \leq
    c\lambda^2$.
  \end{step}

  Because of \autoref{Thm_DSHaydysEq}, the fact that $F^{0,2}_{I_{t,\lambda}}$
  is anti-self-dual and \autoref{Prop_PsiExpansion} we can write $\tilde
  e_{t,\lambda}$ on $V_{[0,\sigma);t}$ as
  \begin{multline*}
    \tilde e_{t,\lambda}
    =*\left[\(F_{I_{t,\lambda}} - F_{B_t|_{P_t}}\)^{2,0}\wedge\psi_t\right] 
    +*\left[F_{I_{t,\lambda}}^{1,1}\wedge(\psi_{1;t}+\psi_{\geq 2;t})\right] \\
    +*\(F_{I_{t,\lambda}}^{0,2}\wedge \psi_{\geq 2;t}\)
    +(*-*_0)\(F_{I_{t,\lambda}}^{1,1}\wedge \psi_{0;t}\).
  \end{multline*}
  Using \autoref{Prop_PsiExpansion} and \autoref{Prop_F} as well as
  $\|1\|_{-1,V_{[0,\sigma);t}}\leq c$ we estimate $\|\tilde
  e_{t,\lambda}\|_{-2,V_{[0,\sigma);t}}$ by
  \begin{align*}
   &\left\|\(F_{I_{t,\lambda}}-F_{B_t|_{P_t}}\)^{2,0}\right\|_{-2,V_{[0,\sigma);t}}
         \cdot \|\psi_t\|_{0,V_{[0,\sigma);t}}
  \\&\qquad
    +\left\|F_{I_{t,\lambda}}^{1,1}\right\|_{-3,V_{[0,\sigma);t}}
     \cdot\(\|\psi_{1;t}\|_{1,V_{[0,\sigma);t}}+\|1\|_{-1,V_{[0,\sigma);t}}
            \cdot\|\psi_{\geq 2;t}\|_{2,V_{[0,\sigma);t}}\)
  \\&\qquad
    +\left\|F_{I_{t,\lambda}}^{0,2}\right\|_{-4,V_{[0,\sigma);t}}
     \cdot\|\psi_{\geq 2;t}\|_{2,V_{[0,\sigma);t}}
  \\&\qquad
    +\|*-*_0\|_{1,V_{[0,\sigma);t}}
     \cdot\left\|F_{I_{t,\lambda}}^{1,1}\right\|_{-3,V_{[0,\sigma);t}}
     \cdot\|\psi_{0;t}\|_{0,V_{[0,\sigma);t}}
    \leq c \lambda^2. 
  \end{align*}
  This proves the assertion.

  \begin{step}\label{Step_ee2}
    We prove that $\|e_{t,\lambda}-\tilde
    e_{t,\lambda}\|_{V_{[0,\lambda);t}}\leq c\lambda^2$.
  \end{step}

  Since
  \begin{equation*}
    \left\|F_{B_t|_{P_t}} \wedge \psi_t\right\|_{-2,V_{[0,\lambda);t}}
    \leq \|1\|_{-2,V_{[0,\lambda);t}} \cdot
         \left\|F_{B_t|_{P_t}} \wedge \psi_t\right\|_{0,V_{[0,\lambda);t}} 
    \leq c\lambda^2,
  \end{equation*}
  it suffices to estimate $F_{A_{t,\lambda}} - F_{I_{t,\lambda}}$ in
  $V_{[0,\lambda);t}$.  Now, in $V_{[0,\lambda);t}$ the curvature of
  $A_{t,\lambda}$ is given by
  \begin{equation*}
    F_{A_{t,\lambda}}
    = F_{I_{t,\lambda}}
    + \chi^-_{t,\lambda} \rd_{I_{t,\lambda}} b_t
    + \frac12 (\chi^-_{t,\lambda})^2 [b_t\wedge b_t]
    + \rd \chi^-_{t,\lambda} \wedge b_t.
  \end{equation*}
  Using \autoref{Prop_iae} and the fact that the cut-off
  functions $\chi_{t,\lambda}^-$ where constructed so that
  $\|\chi^-_{t,\lambda}\|_{0,V_{[0,\sigma)}} +
  \|\rd\chi^-_{t,\lambda}\|_{-1,V_{[0,\sigma)}}\leq c$ we obtain
  \begin{align*}
    &\|F_{A_{t,\lambda}}-F_{I_{t,\lambda}}\|_{-2,V_{[0,\lambda);t}} \\
    &\qquad
    \leq 
       \|1\|_{-2,V_{[0,\lambda);t}} \cdot
       \|\chi^-_{t,\lambda}\|_{0,V_{[0,\lambda);t}} \cdot
       \|\rd_{B_t|_{P_t}} b_t\|_{0,V_{[0,\lambda);t}} \\
    &\qquad\quad 
    +\|\chi^-_{t,\lambda}\|_{0,V_{[0,\lambda);t}} \cdot
     \|i_{t,\lambda}\|_{-3,V_{(\lambda,\sigma);t}} \cdot
     \|b_t\|_{1,V_{[0,\lambda);t}} \\
    &\qquad\quad
    +\frac12
     \|1\|_{-4,V_{[0,\lambda);t}} \cdot 
     \|\chi^-_{t,\lambda}\|_{0,V_{[0,\lambda);t}}^2 \cdot
     \|b_t\|_{1,V_{[0,\lambda);t}}^2     \\
    &\qquad\quad
    +\|1\|_{-2,V_{[0,\lambda);t}} \cdot
    \|\rd\chi^-_{t,\lambda}\|_{-1,V_{[0,\lambda);t}} \cdot
     \|b_t\|_{1,V_{[0,\lambda);t}} 
    \leq c \lambda^2.
  \end{align*}
 
  \begin{step}
    We prove that $\|e_{t,\lambda}-\tilde
    e_{t,\lambda}\|_{V_{(\lambda,\sigma/2);t}}\leq c\lambda^2$.
  \end{step}

  This is an immediate consequence of $F_{B_t}\wedge\psi_t=0$ and
  \autoref{Prop_iae} since in $V_{[\lambda,\sigma/2);t}$ the
  curvature of $A_{t,\lambda}$ is given by
  $F_{A_{t,\lambda}}=F_{B_t}+[i_{t,\lambda}\wedge b_t]+F_{I_{t,\lambda}}-F_{B_t|_{P_t}}$.

  \begin{step}\label{Step_ee4}
    We prove that $\|e_{t,\lambda}-\tilde
    e_{t,\lambda}\|_{V_{[\sigma/2,\sigma);t}}\leq c\lambda^2$.
  \end{step}

  In $V_{[\sigma/2,\sigma);t}$ the curvature of $A_{t,\lambda}$ is
  given by
  \begin{equation*}
    F_{A_{t,\lambda}}
    = F_{B_t}
    + \chi^+_t  \rd_{B_t} i_{t,\lambda}
    + \frac12(\chi^+_t)^2 [i_{t,\lambda}\wedge i_{t,\lambda}]
    + \rd \chi^+_t \wedge i_{t,\lambda}.
  \end{equation*}
  Since $\|\chi^+_t\|_{\ell,V_{[\sigma/2,\sigma);t}} +
  \|\rd\chi^+_t\|_{\ell,V_{[\sigma/2,\sigma);t}} \leq c$, it follows
  that
  \begin{align*}
    &\|F_{A_{t,\lambda}}-F_{B_t}\|_{-2,V_{[\sigma/2,\sigma);t}} \\
    &\qquad
    \leq 
      \|\chi^+_t\|_{2,V_{[\sigma/2,\sigma);t}} \cdot
      \|\rd_{I_{t,\lambda}}i_{t,\lambda}\|_{-4,V_{[\sigma/2,\sigma);t}}\\
    &\qquad\quad
    +\|\chi^+_t\|_{0,V_{[\sigma/2,\sigma);t}} \cdot
     \|b_t\|_{1,V_{[\sigma/2,\sigma);t}} \cdot
     \|i_{t,\lambda}\|_{-3,V_{[\sigma/2,\sigma);t}}\\
    &\qquad\quad
    +\frac12\|\chi^+_t\|_{2,V_{[\sigma/2,\sigma);t}}^2 \cdot
     \|i_{t,\lambda}\|_{-3,V_{[\sigma/2,\sigma);t}}^2\\
    &\qquad\quad
    +\|\rd\chi^+_t\|_{1,V_{[\sigma/2,\sigma);t}} \cdot
     \|i_{t,\lambda}\|_{-3,V_{[\sigma/2,\sigma);t}} \leq c\lambda^2.
  \end{align*}
  We are thus left with estimating
  \begin{equation*}
    \|\iota_{t,\lambda} \hat v \circ \fI_{t} - \hat v \circ
    \fI_{t,\lambda}\|_{-2,V_{[\sigma/2,\sigma);t}}
    \leq c \|\chi_t^+-1\|_{1,V_{[\sigma/2,\sigma);t}} \cdot \|\hat v \circ
    \fI_{t,\lambda}\|_{-3,V_{[\sigma/2,\sigma);t}}.
  \end{equation*}
  To conclude the proof we observe that
  $\|\chi_t^+-1\|_{1,V_{[\sigma/2,\sigma);t}}\leq c$ and that 
  \begin{equation}\label{Eq_hj}
    \|\hat\fI_\lambda\|_{C^{k,\alpha}_{-3,0;t,\lambda}(V_{[0,\sigma);t})} \leq c
    \lambda^2 \|\hat\fI\|_{C^{k,\alpha}}
  \end{equation}
  as a consequence of \autoref{Prop_Decay} and a simple
  scaling consideration.
\end{proof}


\section{Linear estimates}
\label{Sec_LinearAnalysis}

We denote by $\fX_{t,\lambda}$ and $\fY_{t,\lambda}$ the Banach spaces
$C^{1,\alpha}\oplus\R$ and $C^{0,\alpha}\oplus\R$
equipped with the norms
\begin{align*}
  \|(\ua,\eta)\|_{\fX_{t,\lambda}}
  &:=\lambda^{-\delta/2}\|\rho_{t,\lambda}\ua\|_{C^{1,\alpha}_{-1,\delta;t,\lambda}}
   +\lambda\|\pi_{t,\lambda}\ua\|_{C^{1,\alpha}}
   +\lambda|\eta| \qand \\
  \|(\ua,\eta)\|_{\fY_{t,\lambda}}
  &:=\lambda^{-\delta/2}\|\rho_{t,\lambda}\ua\|_{C^{0,\alpha}_{-2,\delta;t,\lambda}}
   +\lambda\|\pi_{t,\lambda}\ua\|_{C^{0,\alpha}}
   +\lambda|\eta|,
\end{align*}
respectively.
Here we fixed $\delta\in(-1,0)$ and $0<\alpha\ll|\delta|$.
For concreteness one may take $\delta=-\frac12$ and $\alpha=\frac{1}{256}$.
It will become apparent in the course of this section that the choice of the relative weights between terms involving $\rho_{t,\lambda}\ua$ and those involving $\pi_{t,\lambda}\ua$ is not completely unnatural.
We consider the linear operator $\bL_{t,\lambda}\co \fX_{t,\lambda}\to \fY_{t,\lambda}$ defined by
\begin{equation*}
  \bL_{t,\lambda}(\ua,\eta)
  :=\(L_{t,\lambda}+\eta\cdot\iota_{t,\lambda}\hat v\circ\fI_t,
      \<\pi_{t,\lambda}\ua,\hat v\circ\fI_{t}\>\).
\end{equation*}

\begin{prop}\label{Prop_inverse}
  For all $t\in(-T',T')$ and $\lambda\in(0,\Lambda]$ the linear
  operator $\bL_{t,\lambda}$ is invertible, $\bL_{t,\lambda}^{-1}$
  depends continuously differentiably on $t$ and continuously on
  $\lambda$ and, moreover, there exists a constant $c>0$, which is
  independent of $t\in(-T',T')$ and $\lambda\in(0,\Lambda]$, such that
  \begin{align}
    \label{Eq_l}
    \|\bL_{t,\lambda}^{-1}(\ub,\zeta)\|_{\fX_{t,\lambda}}
    &\leq  c\|(\ub,\zeta)\|_{\fY_{t,\lambda}} \quad\text{and} \\
    \label{Eq_dl}
    \|\nabla_t\bL_{t,\lambda}^{-1}(\ub,\zeta)\|_{\fX_{t,\lambda}} &\leq
    c\|(\ub,\zeta)\|_{\fY_{t,\lambda}}.
  \end{align}
\end{prop}

The key to this proposition is the following estimate which we will
prove in the course of this section.

\begin{prop}\label{Prop_linear_estimate}
  There exists a constant $c>0$, which is independent of
  $t\in(-T',T')$ and $\lambda\in(0,\Lambda]$, such that
  \begin{align}
    \label{Eq_l0}
    \|(\ua,\eta)\|_{\fX_{t,\lambda}}
    &\leq  c\|\bL_{t,\lambda}(\ua,\eta)\|_{\fY_{t,\lambda}}.
  \end{align}
\end{prop}

\begin{prop}\label{Prop_dle}
  The family of operators $\bL_{t,\lambda}\co \fX_{t,\lambda}\to
  \fY_{t,\lambda}$ depends continuously differentiably on $t$ and
  continuously on $\lambda$ and there exists a constant $c>0$ such
  that for all $t\in(-T',T')$ and $\lambda\in(0,\Lambda]$ we have
 \begin{align*}
    \|\bL_{t,\lambda}(\ua,\eta)\|_{\fY_{t,\lambda}}
    &\leq c\|(\ua,\eta)\|_{\fX_{t,\lambda}} \quad\text{and} \\   
    \|\nabla_t\bL_{t,\lambda}(\ua,\eta)\|_{\fY_{t,\lambda}}
    &\leq c\|(\ua,\eta)\|_{\fX_{t,\lambda}}.
  \end{align*}
\end{prop}

\begin{proof}[Proof of \autoref{Prop_inverse}]
  By \autoref{Prop_linear_estimate} the operator
  $\bL_{t,\lambda}$ is injective and has closed range.  Hence, we can
  identify its cokernel with the kernel of $\bL_{t,\lambda}^*$.  Since
  $\bL_{t,\lambda}$ is formally self-adjoint, it follows from elliptic
  regularity that the kernel of $\bL_{t,\lambda}^*$ agrees with the
  kernel of $\bL_{t,\lambda}^*$ and thus is trivial.  Therefore,
  $\bL_{t,\lambda}$ is invertible.  Now, \eqref{Eq_l} follows at once
  from \eqref{Eq_l0}.  Since $\bL_{t,\lambda}$ depends continuously
  differentiably on $t$ and continuously on $\lambda$, so does
  $\bL_{t,\lambda}$.  Since
  $\nabla_t\bL_{t,\lambda}^{-1}=-\bL_{t,\lambda}^{-1}\nabla_t\bL_{t,\lambda}\bL_{t,\lambda}^{-1}$,
  \eqref{Eq_dl} follows from \eqref{Eq_l} and \autoref{Prop_dle}.
\end{proof}

\subsection{The model operator on \texorpdfstring{$\R^7$}{R7}}
\label{Sec_Model}

Let $I$ be a finite energy ASD instanton on a $G$--bundle $E$ over $\R^4$.
By a slight abuse of notation we denote the pullbacks of $I$
and $E$ to $\R^7=\R^3\oplus\R^4$ by $I$ and $E$ as well.
We define $L_I\co \Omega^0(\R^7,\fg_E)\oplus\Omega^1(\R^7,\fg_E) \to \Omega^0(\R^7,\fg_E)\oplus\Omega^1(\R^7,\fg_E)$ by
\begin{equation*}
  \begin{pmatrix}
    0 & \rd_I^* \\
    \rd_I & *(\psi_0\wedge\rd_I)
  \end{pmatrix}
\end{equation*}
where
\begin{gather*}
  \psi_0:=\frac12 \omega_1\wedge\omega_1
      -\rd x^{23}\wedge\omega_1
      -\rd x^{31}\wedge\omega_2
      -\rd x^{12}\wedge\omega_3 \qand \\
  \omega_1:=\rd x^{45}+\rd x^{67}, \quad
  \omega_2:=\rd x^{46}-\rd x^{47} \quad\text{and}\quad
  \omega_3:=\rd x^{47}+\rd x^{56}.
\end{gather*}
Denote by $\pi_{\R^4}\co\R^3\oplus\R^4\to\R^4$ the projection onto the second summand and define weight functions
\begin{align}\label{Eq_w}
  w(x) := 1 + |\pi_{\R^4}(x)| \qandq
  w(x,y) := \min\{w(x),w(y)\}.
\end{align}
For a Hölder exponent $\alpha\in(0,1)$ and a weight parameter $\beta\in\R$ we define
\begin{align*}
  [f]_{C^{0,\alpha}_{\beta}(U)}
  &:= \sup_{d(x,y) \leq w(x,y)}
        w(x,y)^{\alpha-\beta} \frac{|f(x)-f(y)|}{d(x,y)^\alpha}, \\
  \|f\|_{L^{\infty}_{\beta}(U)}
  &:=\|w^{-\beta}f\|_{L^\infty(U)} \qand \\
  \|f\|_{C^{k,\alpha}_{\beta}(U)}
  &:= \sum_{j=0}^k \|\nabla^j f\|_{L^{\infty}_{\beta-j}(U)}
                + [\nabla^j f]_{C^{0,\alpha}_{\beta-j}(U)}.
\end{align*}
Here $f$ is a section of a vector bundle over $U\subset\R^7$ equipped with an inner product and a compatible connection.
We use parallel transport to compare the values of $f$ at different points.
If $U$ is not specified, then we take $U=\R^7$.
We denote by $C^{k,\alpha}_{\beta}$ the subspace of elements $f$ of the Banach space $C^{k,\alpha}$ with $\|f\|_{C^{k,\alpha}_\beta}<\infty$ equipped with the norm $\|\cdot\|_{C^{k,\alpha}_\beta}$.

The linear operators $L_I$ can serve as a model for $L_{t,\lambda}$ in the following sense:
Fix $t\in(0,T']$ and $x\in P_t$.
Set $I:=I(\fI)|_{N_xP_t}$ and $E:=E(\fI_t)|_{N_xP_t}$.
Identify $T_xY=T_xP_t\oplus N_xP_t$ with $\R^7=\R^3\oplus\R^4$ in such a way that the summands are preserved and $\psi_t|_{T_xY}$ is identified with $\psi_0$.
For $\epsilon_1,\epsilon_2>0$ we define $V_{\epsilon_1,\epsilon_2;t}$ to be the open set which under the exponential map based at $x$ is identified with $\tilde U_{\epsilon_1,\epsilon_2} := B_{\epsilon_1}(0) \times B_{\epsilon_2}(0) \subset \R^3\oplus \R^4$.
With respect to this identification a $\fg_{E_{t,\lambda}}$--valued tensor field $f$ on $V_{\epsilon_1,\epsilon_2;t}$ is identified with a $s_{1/\lambda}^*\fg_E$--valued tensor field $\tilde f$ on $\tilde U_{\epsilon_1,\epsilon_2;\lambda}$, and if $k\in\N$ is a scaling parameter, then with $f$ we can associate a $\fg_E$--valued tensor field $s_{k,\lambda}f$ on $U_{\epsilon_1,\epsilon_2;\lambda}:=\lambda^{-1}\tilde U_{\epsilon_1,\epsilon_2}$ defined by
\begin{equation*}
  (s_{k,\lambda} f)(x,y) := \lambda^k \tilde f(\lambda x, \lambda y).
\end{equation*}

\begin{prop}\label{Prop_R7ModelCompare}
  There is are constants $c,\epsilon_0>0$ such that for $\epsilon\in(0,\epsilon_0]$, $t\in(-T',T')$ and $\lambda\in(0,\Lambda]$ we  have
  \begin{gather*}
    \frac1c    
    \|s_{k,\lambda}f\|_{L^\infty_{\ell+\delta}\(U_{\epsilon,\sqrt\lambda;\lambda_i}\)}\leq
     \lambda^{k+\ell}\|f\|_{L^\infty_{\ell,\delta;t,\lambda}\(V_{\epsilon,\sqrt\lambda;t}\)}
     \leq c
    \|s_{k,\lambda}f\|_{L^\infty_{\ell+\delta}\(U_{\epsilon,\sqrt\lambda;\lambda_i}\)}, \\
    \frac1c    
    \|s_{k,\lambda}f\|_{C^{k,\alpha}_{\ell+\delta}\(U_{\epsilon,\sqrt\lambda;\lambda_i}\)}\leq
     \lambda^{k+\ell}\|f\|_{C^{k,\alpha}_{\ell,\delta;t,\lambda}\(V_{\epsilon,\sqrt\lambda;t}\)}
     \leq c
    \|s_{k,\lambda}f\|_{C^{k,\alpha}_{\ell+\delta}\(U_{\epsilon,\sqrt\lambda;\lambda_i}\)}
 \intertext{and}
    \left\|L_{t,\lambda} \ua - s_{2,\lambda}^{-1} L_I s_{1,\lambda}
      \ua\right\|_{C^{0,\alpha}_{-2,\delta;t,\lambda}\(V_{\epsilon,\sqrt\lambda;t}\)}
    \leq c(\epsilon+\sqrt\lambda)
    \left\|\ua\right\|_{C^{1,\alpha}_{-1,\delta;t,\lambda}\(V_{\epsilon,\sqrt\lambda;t}\)}.
  \end{gather*}
\end{prop}

To better understand $L_I$ it is useful to rewrite it as follows.

\begin{prop}
  \label{Prop_LISquared}
  If we identify $T^*\R^3$ with $\Lambda^+$ via $\rd x^i\mapsto-\omega_i$ and accordingly $\Omega^0(\R^7,\frg_E)\oplus\Omega^1(\R^7,\frg_E)$ with $\Omega^0(\R^7,\(\R\oplus T^*\R^3\oplus T^*\R^4\)\otimes\frg_E)$, then the linear operator $L_I$ can be written as $L_I=F+D_I$ where
  \begin{align*}
    F(\xi,\omega,a) &:= \sum_{i=1}^3 \(-\<\del_i\omega, \omega_i\>,
    \del_i \xi \cdot \omega_i, *_4(\del_i a\wedge\omega_i)\), \\
    D_I&:=\begin{pmatrix}
      0 & \delta_I \\
      \delta_I^* & 0
    \end{pmatrix} \qand
  \end{align*}
  $\delta_I\co\Omega^1(\R^4,\frg_E)\to\Omega^0(\R^4,\frg_E)\oplus\Omega^+(\R^4,\frg_E)$ is as defined in \eqref{Eq_deltaI}.
  Moreover,
  \begin{equation*}
    L_I^*L_I = \Delta_{\R^3}  +
    \begin{pmatrix}
      \delta_I\delta_I^* & \\
      & \delta_I^*\delta_I
    \end{pmatrix}.
  \end{equation*}
\end{prop}

\begin{proof}
  This is easy to verify by a straight-forward computation.
\end{proof}

Let us now recall a key result from \cite[Appendix A]{Walpuski2011}.

\begin{definition}
  A Riemannian manifold $X$ is said to be of \emph{bounded geometry} if it is complete, its Riemann curvature tensor is bounded from above and its injectivity radius is bounded from below.
  A vector bundle over $X$ is said to be of \emph{bounded geometry} if it has trivialisations over balls of a fixed radius such that the transitions functions and all of their derivatives are uniformly bounded.
  We say that a complete oriented Riemannian manifold $X$ has \emph{subexponential volume growth} if for each $x\in X$ the function $r\mapsto\vol\(B_r(x)\)$ grows subexponentially as $r\to\infty$.
\end{definition}

\begin{lemma}\label{Lem_Liouville}
  Let $E$ be a vector bundle of bounded geometry over a Riemannian manifold $X$ of bounded geometry with subexponential volume growth and suppose that $D\co C^\infty(X,E)\to{}C^\infty(X,E)$ is a uniformly elliptic operator of second order whose coefficients and their first derivatives are uniformly bounded, that is non-negative, i.e., $\<Da,a\>\geq 0$ for all $a \in W^{2,2}(X,E)$, and formally self-adjoint.
  If $a \in C^\infty(\R^n\times X,E)$ satisfies
  \begin{equation*}
    (\Delta_{\R^n}+D)a=0
  \end{equation*}
  and $\|a\|_{L^\infty}$ is finite, then $a$ is constant in the $\R^n$--direction, that is $a(x,y)=a(y)$. 
  Here, by slight abuse of notation, we denote the pullback of $E$ to $\R^n\times X$ by $E$ as well.
\end{lemma}

\begin{remark}
  The statement in \cite[Appendix~A]{Walpuski2011} also requires $\|\nabla a\|_{L^\infty}$ to be finite.
  This, however, can be deduced from $\|a\|_{L^\infty}<\infty$, elliptic estimates and the equation $(\Delta_{\R^n}+D)a=0$.
\end{remark}

\begin{cor}
  \label{Cor_LITranslationInvariantKernel}
  If $\ua\in \Omega^0(\R^7,\frg_E)\oplus\Omega^1(\R^7,\frg_E)$ satisfies $L_I=0$ and $\|\ua\|_{L^\infty}$ is finite, then $\ua$ is the pullback of an element in the kernel of $\delta_I$.
\end{cor}

\begin{prop}
  \label{Prop_ModelSchauderEstimate}
  For $\beta\in\R$ there is a constant $c=c(I)>0$ depending continuously on $I$ such that the following estimate holds
  \begin{equation*}
    \|\ua\|_{C^{1,\alpha}_\beta}
    \leq c\(\|L_I \ua\|_{C^{0,\alpha}_{\beta-1}}
        + \|\ua\|_{L^\infty_\beta}\).
  \end{equation*}
\end{prop}

\begin{proof}
  This is a standard result.
  The argument we use goes back to work of Nirenberg--Walker \cite[Theorem 3.1]{Nirenberg1973}.

  The desired estimate is local in the sense that is enough to prove estimates of the form
  \begin{equation*}
    \|\ua\|_{C^{1,\alpha}_{\beta}(U_i)} \leq c\big(\|L_I
    \ua\|_{C^{0,\alpha}_{\beta-1}} + \|\ua\|_{L^{\infty}_{\beta}}\big)
  \end{equation*}
  with $c>0$ independent of $i$, where $\{U_i\}$ is a suitable open cover of $\R^3\times X$.

  Fix $R>0$ suitably large and set $U_0:=\set{ (x,y)\in\R^3\times X : |\pi_{\R^4}(x)|\leq R }$.
  Then there clearly is a constant $c>0$ such that the above estimate holds for $U_i=U_0$.
  Pick a sequence $(x_i,y_i)\in\R^3\times{}X$ such that $r_i:=|\pi_{\R^4}(y_i)|\geq R$ and the balls $U_i:=B_{r_i/8}(x_i,y_i)$ cover the complement of $U_0$.
  On $U_i$, we have a Schauder estimate of the form
  \begin{gather*}
    \|\underline{a}\|_{L^\infty(U_i)} +
    r_i^{\alpha}[\underline{a}]_{C^{0,\alpha}(U_i)} + r_i\|\nabla_I
    \underline{a}\|_{L^\infty(U_i)} +
    r_i^{1+\alpha}[\nabla_I \underline{a}]_{C^{0,\alpha}(U_i)} \\
    \leq c\left(r_i\|L_I\underline a\|_{L^\infty(V_i)} +
      r_i^{1+\alpha}[L_I\underline{a}]_{C^{0,\alpha}(V_i)} +
      \|\underline{a}\|_{L^\infty(V_i)}\right)
  \end{gather*}
  where $V_i=B_{r_i/4}(x_i,y_i)$ and $\underline{a}=(\xi,a)$.
  By rescaling the balls $V_i$ to a ball of fixed radius one can see that the constant $c>0$ can be chosen to work for all $i$ simultaneously.
  Since on $V_i$ we have $\frac12 r_i \leq w \leq 2r_i$, multiplying the above Schauder estimate by $r_i^{-\beta}$ yields the desired local estimate.
\end{proof}

\subsection{Schauder estimate}

\begin{prop}\label{Prop_schauder}
  There is a constant $c>0$ such that for all $t\in(-T',T')$ and
  $\lambda\in(0,\Lambda]$ the following estimate holds
  \begin{equation}
    \|\ua\|_{C^{1,\alpha}_{-1,\delta;t,\lambda}}
    \leq c \(\|L_{t,\lambda}\ua\|_{C^{0,\alpha}_{-2,\delta;t,\lambda}}
         + \|\ua\|_{L^{\infty}_{-1,\delta;t,\lambda}}\).
  \end{equation}
\end{prop}

\begin{proof}
  It suffices to show that there is a constant $c>0$ such that for all
  $t\in(-T',T')$, $\lambda\in(0,\Lambda]$ and $x\in Y$ there
  exist open sets $U$ and $V$ such that
  \begin{equation*}
    \|\ua\|_{C^{1,\alpha}_{-1,\delta;t,\lambda}(U)}
    \leq c \(\|L_{t,\lambda}\ua\|_{C^{0,\alpha}_{-2,\delta;t,\lambda}(V)}
           + \|\ua\|_{L^{\infty}_{-1,\delta;t,\lambda}(V)}\).
  \end{equation*}
  For $x\in Y$ with $r_t(x)\leq \sqrt\lambda$ such an estimate follows
  from \autoref{Prop_R7ModelCompare} and \autoref{Prop_ModelSchauderEstimate}.  For $x\in Y$
  with $r_t(x)> \sqrt\lambda$ one can take $U=B_{r_t(x)/8}(x)$ and
  $V=B_{r_t(x)/4}(x)$ and argue as in the proof of \autoref{Prop_ModelSchauderEstimate}.
\end{proof}

\subsection{Estimate of \texorpdfstring{$\|\rho_{t,\lambda}$\underline{$a$}$\|_{L^{\infty}_{-1,\delta;t,\lambda}}$}{...}}

\begin{prop}\label{Prop_X}
  There is a constant $c>0$ such that for all $t\in(-T',T')$ and
  $\lambda\in(0,\Lambda]$ the following estimate holds
  \begin{equation}
    \|\ua\|_{L^{\infty}_{-1,\delta;t,\lambda}}
    \leq c\(\|L_{t,\lambda}\ua\|_{C^{0,\alpha}_{-2,\delta;t,\lambda}}
          + \|\bar\pi_{t,\lambda}\ua\|_{L^{\infty}_{-1,\delta;t,\lambda}}\).
  \end{equation}
\end{prop}

\begin{proof}
  If not, then there exist sequences $(t_i)$, $(\lambda_i)$ and
  $(\ua_i)$ such that $\lim_{i\to\infty}\lambda_i=0$,
  \begin{align}
    \nonumber
    \|\ua_i\|_{L^\infty_{-1,\delta;t_i,\lambda_i}}&=1, \\
    \nonumber
    \lim_{i\to\infty}\|L_{t_i,\lambda_i}\ua_i\|_{{C^{0,\alpha}_{-2,\delta;t_i,\lambda_i}}}&=0
    \qand \\
    \label{Eq_pi0}
    \lim_{i\to\infty}\|\bar\pi_{t_i,\lambda_i}\ua_i\|_{L^{\infty}_{-1,\delta;t_i,\lambda_i}}&=0.
  \end{align}
  After passing to a subsequence we can assume that $(t_i)$ converges
  to a limit $t$.  From \autoref{Prop_schauder} it follows that
  \begin{equation}\label{Eq_1}
    \|\ua_i\|_{C^{1,\alpha}_{-1,\delta;t_i,\lambda_i}} \leq c.
  \end{equation}
  Pick a sequence $(x_i)$ of points in $Y$ such that
  \begin{equation*}
    w_{-1,\delta;t_i,\lambda_i}(x_i)|\ua_i|(x_i)=1.
  \end{equation*}  
  After passing to a subsequence we can assume that one of the
  following cases occurs.  We will show that each of them leads to a
  contradiction, thus proving the proposition.
 
  \setcounter{case}{0}
  \begin{case}
    The sequence $(x_i)$ accumulates away from $P_t$:
    $\lim_{i\to\infty} r_{t_i}(x_i) > 0$.
  \end{case}
  \noindent
  By \eqref{Eq_1} the sequence $(\ua_i)$ is uniformly bounded in
  $C^{1,\alpha}$ on each compact subset of $Y\setminus P_t$.
  Arzelà--Ascoli and a diagonal sequence argument thus yield a
  subsequence of $(\ua_i)$ which converges to a limit $\ua$ on
  $Y\setminus P_t$ in $C^{1,\alpha/2}_\loc$.  Since we can also
  arrange that the corresponding subsequence of $(x_i)$ converges to a
  limit $x\in Y\setminus P_t$ for which $r_t(x)^{1+\delta}|\ua|(x)=1$,
  it follows that $\ua$ cannot vanish identically.  However, $\ua$
  also satisfies
  \begin{align}
    \label{Eq_2}
    \|r_t^{1+\delta}\ua\|_{L^\infty} &\leq 1 \qand \\
    \label{Eq_3}
    L_{B_t,\phi_t}\ua&=0 \quad\text{on}~Y\setminus P_t.
  \end{align}
  Since $\delta<2$, it follows from \eqref{Eq_2} that $\ua$ satisfies
  \eqref{Eq_3} on all of $Y$ in the sense of distribution and hence is
  smooth by elliptic regularity.  This contradicts the hypothesis that
  $A_0$ and hence $B_t$ is acyclic, i.e., that
  $L_{B_t,\phi_t}$ has trivial kernel.

  \begin{case}
    The sequence $(x_i)$ quickly accumulates near $P_t$:
    $\lim_{i\to\infty} r_{t_i}(x_i)/\lambda_i<\infty$.
  \end{case}
  \noindent
  After passing to a subsequence we can assume that $(x_i)$ converges
  to a point $x\in P_t$.  With the notation of the paragraph preceding
  \autoref{Prop_R7ModelCompare}, set
  \begin{equation*}
    \ub_i:=s_{1,\lambda_i}\(\ua_i|_{V_{\sqrt\lambda_i,\sqrt\lambda_i;t}}\).
  \end{equation*}
  This sequence satisfies
  \begin{equation*}
    \|\ub_i\|_{C^{1,\alpha}_{-1+\delta}\(U_{\sqrt\lambda_i,\sqrt\lambda_i;\lambda_i}\)}\leq
    c \qandq
    \lim_{i\to\infty}\|L_I\ub_i\|_{{C^{0,\alpha}_{-2+\delta}}}=0,
  \end{equation*}
  and if $(y_i)$ denotes the sequence of points in
  $U_{\sqrt\lambda_i,\sqrt\lambda_i;\lambda_i}$ corresponding to the
  sequence $(x_i)$, then
  \begin{equation*}
    w(y_i)^{1-\delta}|\ub_i|(y_i)\geq \frac12.
  \end{equation*} 
  where $w=1+|\pi_{\R^4}|$ as in \eqref{Eq_w}.  Since the sequence of
  subsets $U_{\sqrt\lambda_i,\sqrt\lambda_i;\lambda_i}\subset \R^7$ is
  exhaustive, Arzelà--Ascoli and a diagonal sequence argument yield a
  subsequence of $(\ub_i)$ which converges to a limit $\ub$ on $\R^7$
  in $C^{1,\alpha/2}_\loc$.  By translation we can arrange that the
  $\R^3$--component of $y_i$ is zero and thus $|y_i|$ is bounded.
  After passing to a further subsequence $(y_i)$ converges to a limit
  $y\in\R^7$.  At this point we must have $w(y)^{1-\delta}|\ub|(y)\geq
  1/2$ and thus $\ub$ cannot vanish identically.  It follows from
  \autoref{Prop_R7ModelCompare} that $\ub$ satisfies
  \begin{equation*}
    \|w^{1-\delta}\ub\|_{L^\infty}
    \leq 2 \qandq L_I\ub=0.
  \end{equation*}
  Moreover, using \eqref{Eq_pi0} and arguing as in the proof of
  \autoref{Prop_IotaPi}, making use of the hypothesis $\delta<0$,
  one can show that each restriction of $\ub=0$ to a slice
  $\{x\}\times\R^4$ is $L^2$--orthogonal to $\ker\delta_I$.  This,
  however, contradicts \autoref{Cor_LITranslationInvariantKernel}.

  \begin{case}
    The sequence $(x_i)$ slowly accumulates near $P_t$:
    $\lim_{i\to\infty} r_{t_i}(x_i)/\lambda_i=\infty$.
  \end{case}
  \noindent
  In a similar manner as in the previous case we set
  \begin{equation*}
    \ub_i:=s_{1,\lambda_i}\(\ua_i|_{V_{\sqrt\lambda_i,\sigma;t}}\)
  \end{equation*}
  and denote by $(y_i)$ the sequence of points in
  $U_{\sqrt\lambda_i,\sigma;\lambda_i}$.  Again, we can assume that
  the $\R^3$--component of $y_i$ is zero.  After passing to a
  subsequence we can assume that one of the following two cases
  occurs.
 
  \begin{subcase}
    We have $|y_i|\leq 1/\sqrt\lambda_i$ for all $i\in\N$.
  \end{subcase}

  Set
  \begin{equation*}
    \tilde \ub_i:=|y_i|^{1-\delta}\ub_i(|y_i|\cdot -) \qandq
    \tilde y_i:=y_i/|y_i|.
  \end{equation*}
  Again, Arzelà--Ascoli and a diagonal sequence argument yield a
  subsequence of $(\tilde\ub_i)$ converging to a limit $\tilde\ub$ on
  $\R^7\setminus (\R^3\times\{0\})$ which cannot vanish identically,
  since $|\tilde y|^{1-\delta}|\tilde\ua|(\tilde y)\geq 1/4$ with
  $\tilde y:=\lim_{t\to\infty} \tilde y_i$.  However, $\tilde\ub$ also satisfies
  \begin{align}
    \label{Eq_9}
    \|\tilde w^{1-\delta}\tilde\ub\|_{L^\infty}
   &\leq 4 \qand \\
    \label{Eq_4}
    L\tilde\ub&=0 \quad\text{on}~\R^7\setminus (\R^3\times\{0\}).
  \end{align}
  Here $\tilde w:=|\pi_{\R^4}|$ and $L$ is defined by
  \begin{equation}\label{Eq_L0}
    L\ua:=\(\rd^*a, \rd \xi+*(\psi\wedge\rd a)\).
  \end{equation}
  Since $\delta>-2$, it follows from \eqref{Eq_9} that $\tilde\ub$
  solves \eqref{Eq_4} on all of $\R^7$ in the sense of distributions
  and hence is smooth by elliptic regularity.  Moreover, using 
  standard elliptic estimates one can show that $\tilde\ub$ is
  uniformly bounded near $\R^3$ and therefore by \eqref{Eq_9} on all
  of $\R^7$ since $\delta\leq 1$.  Because
  $L^*L=\Delta_{\R^3}+\Delta_{\R^4}$, we can now apply \autoref{Lem_Liouville} to conclude that $\tilde\ub$ is invariant under
  translations in the $\R^4$--direction.  We can thus think of
  $\tilde\ub$ as a vector of harmonic functions on $\R^4$.  Since
  $\delta<1$ it follows that the components of $\tilde\ub$ decay to
  zero at infinity and thus vanish by the maximum principle.  This,
  however, contradicts the fact that $\tilde\ub$ cannot vanish
  identically.

  \begin{subcase}
    We have $|y_i|>1/\sqrt\lambda_i$ for all $i\in\N$.
  \end{subcase}

  If we set
  \begin{equation*}
    \tilde \ub_i:=\lambda_i^\delta|y_i|^{1+\delta}\ub_i(|y_i|\cdot -) \qandq \tilde
    y_i:=y_i/|y_i|,
  \end{equation*}
  then we obtain the desired contradiction by arguing, mutatis
  mutandis, as the previous case.  The relevant constraint on $\delta$
  is easily seen to be that $\delta\in(-1,2)$.
\end{proof}

\subsection{Comparison with \texorpdfstring{$F_{\fI_t}$}{F}}

The connection on $\coprod_{t\in(-T,T)} NP_t$ induces connections on
the bundles over $(-T,T)$ whose fibres are
$C^{0,\alpha}(\fI_{t,\lambda}^*V\fM_t)$ and $C^{1,\alpha}(\fI_{t,\lambda}^*V\fM_t)$,
respectively.  We denote the corresponding covariant derivatives by
$\nabla_t$.

\begin{prop}\label{Prop_cf}
  There is a constant $c>0$ such that for all $t\in(-T',T')$ and
  $\lambda\in(0,\Lambda]$ the following estimate holds
  \begin{align*}
    \|L_{t,\lambda}\iota_{t,\lambda}\hat\fI
      -\iota_{t,\lambda}F_{\fI_t}\hat\fI\|_{C^{0,\alpha}_{-2,0;t,\lambda}}
   &\leq c\lambda^2 \|\hat\fI\|_{C^{1,\alpha}} \qand \\
     \|(\nabla_t L_{t,\lambda})\iota_{t,\lambda}\hat\fI
      -\iota_{t,\lambda} (\nabla_tF_{\fI_t})\hat\fI\|_{C^{0,\alpha}_{-2,0;t,\lambda}}
   &\leq c\lambda^2 \|\hat\fI\|_{C^{1,\alpha}}.  
  \end{align*}
\end{prop}

\begin{proof}
  Consider the operator
  \begin{equation*}
    \tilde L_{t}\co
    \Omega^0(NP_t,\fg_{E_t})\oplus\Omega^1(NP_t,\fg_{E_t}) \to
    \Omega^0(NP_t,\fg_{E_t})\oplus\Omega^1(NP_t,\fg_{E_t})
  \end{equation*}
  defined by
  \begin{equation*}
    \tilde L_{t,\lambda} \ua :=(\rd_{I_{t,\lambda}}^*a,
    \rd_{I_{t,\lambda}}\xi+*_0(\psi_{0;t} \wedge\rd_{I_{t,\lambda}} a)).
  \end{equation*}
  If we identify $\hat\fI\in\Gamma(\fI_t^*V\fM_t)$ with an element of
  $\Omega^1(NP_t,\fg_{E_t})$, then since
  $\delta_{\fI_t(x)}(\hat\fI|_{N_xP_t})=0$ we have
  \begin{equation*}
    \tilde L_{t,\lambda}\hat\fI
    =\(0,*_0\left[\psi_{0;t}\wedge(\rd_{I_{t,\lambda}}\hat\fI)^{1,1}\right]\)
  \end{equation*}
  which is the same as $s_{1/\lambda}^*\circ F_{\fI_t}\circ
  s_{\lambda}^*(\hat\fI)$.

  In order to prove the first estimate it thus suffices to control the
  following terms
  \begin{align*}
    L_{t,\lambda}\iota_{t,\lambda}\hat\fI
    -\iota_{t,\lambda}F_{\fI_t}\hat\fI
   &=L_{t,\lambda} (\iota_{t,\lambda} \hat\fI -
    \hat\fI_\lambda) + (L_{t,\lambda}-\tilde
    L_{t,\lambda}) \hat\fI_\lambda \\
    &\quad +(s_{1/\lambda}^* \circ F_{\fI_t}\hat\fI-\iota_{t,\lambda}F_{\fI_t}\hat\fI) \\
   &=:\rI+\rII+\rIII
  \end{align*} 
  on $V_{[0,\sigma);t}$.  It is easy to see that
  \begin{equation*}
    \|\rI\|_{C^{0,\alpha}_{-2,0;t,\lambda}(V_{[0,\sigma);t})} +
  \|\rIII\|_{C^{0,\alpha}_{-2,0;t,\lambda}(V_{[0,\sigma);t})} \leq
  c\lambda^2\|\hat\fI\|_{C^{1,\alpha}}
  \end{equation*}
  by using that fact that $\rI$ and $\rIII$ are supported in
  $V_{[\sigma/2,\sigma);t}$ and the estimates
  \begin{equation*}
    \|L_{t,\lambda} \ua\|_{C^{0,\alpha}_{-2,0;t,\lambda}(V_{[0,\sigma);t})}\leq c
    \|\ua\|_{C^{1,\alpha}_{-1,0;t,\lambda}(V_{[0,\sigma);t})} \quad\text{and}\quad
    \|F_{\fI_t}\hat\fI\|_{C^{0,\alpha}}\leq c \|\hat\fI\|_{C^{1,\alpha}}.
  \end{equation*}
  as well as
  \begin{align*}
    &\|\iota_{t,\lambda}\hat\fI
      -\hat\fI_\lambda\|_{C^{k,\alpha}_{-\ell,0;t,\lambda}(V_{[\sigma/2,\sigma);t})} \\
   &\quad\leq \|\chi_t^+-1\|_{{C^{k,\alpha}_{\ell+3,0;t,\lambda}}(V_{[\sigma/2,\sigma);t})}
         \cdot\|\hat \fI_\lambda\|_{C^{k,\alpha}_{-3,0;t,\lambda}(V_{[\sigma/2,\sigma);t})} \\
   &\quad\leq c\lambda^2 \|\hat\fI\|_{C^{k,\alpha}}.
  \end{align*}
  To estimate $\rII$ we expand it as
  \begin{align*}
    \rII&=*\(\psi_t\wedge(A_{t,\lambda}-I_{t,\lambda})\wedge\hat\fI_\lambda\)
    +*\((\psi_{1;t}+\psi_{\geq 2;t})\wedge\rd_{I_{t,\lambda}}\hat\fI_\lambda\) \\
    &\qquad+(*-*_0)(\psi_{0;t}\wedge\rd_{I_{t,\lambda}}\hat\fI_\lambda) \\
    &=:\rII_1+\rII_2+\rII_3.
  \end{align*}
  It follows from \autoref{Prop_iae} that
  \begin{equation}\label{Eq_aie}
    \|A_{t,\lambda}-I_{t,\lambda}\|_{C^{0,\alpha}_{1,0;t,\lambda}(V_{[0,\sigma);t})}
    =\|\chi_{t,\lambda}^-b_t
       +(\chi_t^+-1)i_{t,\lambda}\|_{C^{0,\alpha}_{1,0;t,\lambda}(V_{[0,\sigma);t})}
    \leq c
  \end{equation}
  which in conjunction with \eqref{Eq_hj} yields
  \begin{equation*}
    \|\rII_1\|_{C^{0,\alpha}_{-2,0;t,\lambda}} \leq
  c\lambda^2\|\hat\fI\|_{C^{1,\alpha}}.
  \end{equation*}
  From \autoref{Prop_Decay} and simple scaling considerations
  it follows that
  \begin{equation*}
    \|(\rd_{I_{t,\lambda}}\hat\fI_\lambda)^{1,1}\|_{C^{0,\alpha}_{-3,0;t,\lambda}(V_{[0,\sigma);t})}
    +
    \|(\rd_{I_{t,\lambda}}\hat\fI_\lambda)^{0,2}\|_{C^{0,\alpha}_{-4,0;t,\lambda}(V_{[0,\sigma);t})}
    \leq c \lambda^2 \|\hat\fI\|_{C^{1,\alpha}}.
  \end{equation*}
  Since $\delta_{\fI_t(x)}(\hat\fI|_{N_xP_t})=0$, we have
  \begin{equation*}
    \psi_{0;t}\wedge(\rd_{I_{t,\lambda}}\hat\fI_\lambda)^{0,2}=
    \psi_{1;t}\wedge(\rd_{I_{t,\lambda}}\hat\fI_\lambda)^{0,2}=0.
  \end{equation*}
  These facts together with \autoref{Prop_PsiExpansion} imply that
  \begin{equation*}
    \|\rII_2\|_{C^{0,\alpha}_{-2,0;t,\lambda}} +
    \|\rII_3\|_{C^{0,\alpha}_{-2,0;t,\lambda}} \leq
    c\lambda^2\|\hat\fI\|_{C^{1,\alpha}}.
  \end{equation*}
  This finishes the proof of the first estimate.
  To prove the second estimate note that the individual terms of
  \begin{align*}
    (\nabla_tL_{t,\lambda}\iota_{t,\lambda})\hat\fI
    -\nabla_tF_{\fI_t}\hat\fI &=(\nabla_tL_{t,\lambda})
    (\iota_{t,\lambda} \hat\fI -
    \hat\fI_\lambda) + \nabla_t(L_{t,\lambda}-\tilde L_{t,\lambda})\hat\fI_\lambda \\
    &\qquad +(s_{1/\lambda}^*\circ\nabla_t
    F_{\fI_t}\hat\fI-\iota_{t,\lambda}\nabla_tF_{\fI_t}\hat\fI)
  \end{align*}
  can be estimated just as above.
\end{proof}

\begin{prop}\label{Prop_cf2}
  There is a constant $c>0$ such that for all $t\in(-T',T')$ and
  $\lambda\in(0,\Lambda]$ we have
  \begin{equation*}
    \|\hat \fI\|_{C^{1,\alpha}} \leq c
    \(\|\pi_{t,\lambda}L_{t,\lambda}\iota_{t,\lambda}\hat\fI\|_{C^{0,\alpha}}
    + \left|\langle\hat\fI,\hat v\circ \fI_t\rangle\right|\).
  \end{equation*}
\end{prop}

\begin{proof}
  By hypothesis we have
  \begin{equation*}
    \|\hat \fI\|_{C^{1,\alpha}} \leq c
    \(\|F_{\fI_t}\hat\fI\|_{C^{0,\alpha}}
    + \left|\langle\hat\fI,\hat v\circ \fI_t\rangle\right|\)
  \end{equation*}
  for $t=0$ and thus also for $t\in(-T',T')$.  Together with
   \begin{equation*}
     \|F_{\fI_t}\hat\fI\|_{C^{0,\alpha}} \leq
     c\(\|\pi_{t,\lambda}L_{t,\lambda}\iota_{t,\lambda}\hat\fI\|_{C^{0,\alpha}}
     +\lambda^{1-\alpha}\|\hat\fI\|_{C^{1,\alpha}}\),
   \end{equation*}  
   which is an immediate corollary of \autoref{Prop_IotaPi} and \autoref{Prop_cf}, this immediately implies the asserted estimate by rearranging.
\end{proof}

\subsection{Cross-term estimates}

\begin{prop}\label{Prop_cte}
  There is a constant $c>0$ such that for all $t\in(-T',T')$ and
  $\lambda\in(0,\Lambda]$ we have
  \begin{align*}
    \|\rho_{t,\lambda}L_{t,\lambda}\iota_{t,\lambda}\hat\fI\|_{C^{0,\alpha}_{-2,0;t,\lambda}}
    &\leq c\lambda^{2-\alpha} \|\hat\fI\|_{C^{1,\alpha}} \qand \\
    \|\rho_{t,\lambda}(\nabla_t
    L_{t,\lambda})\iota_{t,\lambda}\hat\fI\|_{C^{0,\alpha}_{-2,0;t,\lambda}}
    &\leq c\lambda^{2-\alpha} \|\hat\fI\|_{C^{1,\alpha}}
    \intertext{as well as}
    \|\pi_{t,\lambda}L_{t,\lambda}\rho_{t,\lambda}\ua\|_{C^{0,\alpha}}
    &\leq
    c\lambda^{-\alpha}\|\rho_{t,\lambda}\ua\|_{C^{1,\alpha}_{-1,0;t,\lambda}}
    \qand \\
     \|\pi_{t,\lambda}(\nabla_tL_{t,\lambda})\rho_{t,\lambda}\ua\|_{C^{0,\alpha}}
    &\leq c\lambda^{-\alpha} \|\rho_{t,\lambda}\ua\|_{C^{1,\alpha}_{-1,0;t,\lambda}}.
  \end{align*}
\end{prop}

\begin{proof}
  The first two estimates are immediate consequences of
  \autoref{Prop_IotaPi} and \autoref{Prop_cf} because
  \begin{equation*}
    \rho_{t,\lambda}L_{t,\lambda}\iota_{t,\lambda}\hat\fI
    =\rho_{t,\lambda}(L_{t,\lambda}\iota_{t,\lambda}\hat\fI
    -\iota_{t,\lambda}F_{\fI_t}\hat\fI)
  \end{equation*}
  and similarly for $\nabla_tL_{t,\lambda}$.

  To prove the last two estimates first note that we can assume
  without loss of generality that $\ua$ is supported in
  $V_{[0,\sigma)}$ and that $\ua=\rho_{t,\lambda}\ua$.  Define
  $\tilde\pi_{t,\lambda}:\Omega^1(NP_t,\fg_{E_t}) \to
  \Gamma(\fI_t^*V\fM_t)$ by
  \begin{equation*}
    (\tilde\pi_{t,\lambda} a)(x):=\sum_\kappa \int_{N_xP_t} \<a,s_{1/\lambda}^*\kappa\>
    s_{1/\lambda}^*\kappa
  \end{equation*}
  where, at each point $x\in P_t$, $\kappa$ runs through an
  orthonormal basis of $(V\fM_t)_{\fI_{t}(x)}$ with respect to
  $\<s_{1/\lambda}^*\cdot,s_{1/\lambda}^*\cdot\>$ and set
  $\tilde\rho_{t,\lambda}:=\id-\tilde\pi_{t,\lambda}$.  One can check
  $\tilde\pi_{t,\lambda}a=0$ implies that $\tilde\pi_{t,\lambda}\tilde
  L_{t,\lambda} a =0$ where $\tilde L_{t,\lambda}$ is as defined in
  the proof of \autoref{Prop_cf}.  Therefore
  \begin{align*}
    \pi_{t,\lambda}L_{t,\lambda}\rho_{t,\lambda}\ua
    &=\pi_{t,\lambda} (L_{t,\lambda}-\tilde L_{t,\lambda})\rho_{t,\lambda} \ua
    + (\pi_{t,\lambda}-\tilde\pi_{t,\lambda}) \tilde L_{t,\lambda}\rho_{t,\lambda} \ua \\
    &\quad + \tilde\pi_{t,\lambda} \tilde L_{t,\lambda}(\rho_{t,\lambda}-\tilde\rho_{t,\lambda}) \ua \\
    &=:\pi_{t,\lambda}\rI + \rII + \rIII.
  \end{align*}
  Define $\tilde{\tilde{\pi}}_{t,\lambda}$ like
  $\tilde\pi_{t,\lambda}$ but take the inner product with
  $\iota_{t,\lambda}\kappa$ instead of $s_{1/\lambda}^*\kappa$ and let
  $\kappa$ run through an orthonormal basis with respect to
  $\<\iota_{t,\lambda}\cdot,\iota_{t,\lambda}\cdot\>_{L^2}$.  If
  $\tilde \rII$ and $\tilde \rIII$ denote the same expressions as
  $\rII$ and $\rIII$ but with $\tilde{\tilde{\pi}}_{t,\lambda}$ in
  place of $\tilde\pi_{t,\lambda}$ and
  $\id-\tilde{\tilde{\pi}}_{t,\lambda}$ in place of
  $\tilde\rho_{t,\lambda}$, then $\tilde\rII$ and $\tilde\rIII$ are
  supported in $V_{[\sigma/2,\sigma)}$ and one can argue as in the
  proof of \autoref{Prop_cf} to show
  \begin{equation*}
    \|\tilde\rII\|_{C^{0,\alpha}} + \|\tilde\rIII\|_{C^{0,\alpha}}
    \leq c\lambda^{-\alpha}\|\ua\|_{C^{0,\alpha}_{-1,\delta;t,\lambda}}.
  \end{equation*}
  The eigenvalues of the quadratic form
  $\<\iota_{t,\lambda}\cdot,\iota_{t,\lambda}\cdot\>$ with respect to
  $\<s_{1/\lambda}^*\cdot,s_{1/\lambda}^*\cdot\>$ differ from one by
  $O(\lambda^4)$; hence, the differences between $\rII$ and
  $\tilde\rII$ as well as between $\rIII$ and $\tilde\rIII$ are
  negligibly small.

  To estimate $\rI$ we write it as
  \begin{align*}
    \rI&=*\(\psi_t\wedge(A_{t,\lambda}-I_{t,\lambda})\wedge\rho_{t,\lambda}\ua\)
       +*\((\psi_{1;t}+\psi_{\geq 2;t})\wedge\rd_{I_{t,\lambda}}\rho_{t,\lambda}\ua\) \\
       &\quad+(*-*_0)\(\psi_0\wedge\rd_{I_{t,\lambda}}\rho_{t,\lambda}\ua\).
  \end{align*}
  Using \autoref{Prop_IotaPi}, \autoref{Prop_PsiExpansion} and \autoref{Prop_*Expansion}
  as well as and \eqref{Eq_aie} it follows that
  \begin{equation*}
    \|\pi_{t,\lambda}\rI\|_{C^{0,\alpha}}
    \leq c\lambda^{-\alpha}\|\rI\|_{C^{0,\alpha}_{-1,0;t,\lambda}(V_{[0,\sigma);t})}
    \leq c\lambda^{-\alpha}\|\ua\|_{C^{1,\alpha}_{-1,\delta;t,\lambda}}.
  \end{equation*}
  This finishes the proof of the third estimate.  The last estimate is
  proved along the same lines.
\end{proof}

\subsection{Proof of \autoref{Prop_linear_estimate}}

Applying \autoref{Prop_schauder} and \autoref{Prop_X} to
$\rho_{t,\lambda}\ua$ and using \autoref{Prop_IotaPi} yields
\begin{align*}
  \|\rho_{t,\lambda}\ua\|_{C^{1,\alpha}_{-1,\delta;t,\lambda}}
  &\leq c
  \|L_{t,\lambda}\rho_{t,\lambda}\ua\|_{C^{0,\alpha}_{-2,\delta;t,\lambda}} \\
  &\leq c\big(
 \|\rho_{t,\lambda}L_{t,\lambda}\ua\|_{C^{0,\alpha}_{-2,\delta;t,\lambda}}
    +
    \|\rho_{t,\lambda}L_{t,\lambda}\bar\pi_{t,\lambda}\ua\|_{C^{0,\alpha}_{-2,\delta;t,\lambda}} \\
 &\qquad\qquad +\lambda^{1-\alpha}\|\pi_{t,\lambda}L_{t,\lambda}\rho_{t,\lambda}\ua\|_{C^{0,\alpha}}
 \big).
\end{align*}
By \autoref{Prop_cf2}
\begin{equation*}
  \|\pi_{t,\lambda}\ua\|_{C^{1,\alpha}} \leq
  c\(\|\pi_{t,\lambda}L_{t,\lambda}\ua\|_{C^{0,\alpha}}     
  + \left|\<\pi_{t,\lambda}\ua,\hat v\circ \fI_{t}\>\right|
  +\|\pi_{t,\lambda}L_{t,\lambda}\rho_{t,\lambda}\ua\|_{C^{0,\alpha}}\).
\end{equation*}
Recalling the definitions of $\|\cdot\|_{\fX_{t,\lambda}}$,
$\|\cdot\|_{\fY_{t,\lambda}}$ and $\bL_{t,\lambda}$, and using
\autoref{Prop_cte} it follows that
\begin{equation*}
  \|(\ua,\eta)\|_{\fX_{t,\lambda}}
  \leq c\(\|\bL_{t,\lambda}(\ua,\eta)\|_{\fY_{t,\lambda}}+\lambda^{1-\alpha}
  \|(\ua,\eta)\|_{\fX_{t,\lambda}}\)
\end{equation*}
which yields \eqref{Eq_l0} by rearranging.
\qed

\subsection{Proof of \autoref{Prop_dle}}

It is clear that $L_{t,\lambda}$ depends continuously differentiably
on $t$ and continuously on $\lambda$.  By \autoref{Prop_cf} and \autoref{Prop_cte} as well as \autoref{Cor_DeltaSensitivity} and
\autoref{Prop_IotaPi} we have
\begin{align*}
  \|\rho_{t,\lambda}L_{t,\lambda}\ua\|_{C^{0,\alpha}_{-2,\delta;t,\lambda}}
 &\leq
  c\(\|L_{t,\lambda}\rho_{t,\lambda}\ua\|_{C^{0,\alpha}_{-2,\delta;t,\lambda}}
  +\lambda^{1-\alpha}\|\pi_{t,\lambda}L_{t,\lambda}\rho_{t,\lambda}\ua\|_{C^{0,\alpha}}\right. \\
 &\qquad\qquad\left.
  +\|\rho_{t,\lambda}L_{t,\lambda}\bar\pi_{t,\lambda}\ua\|_{C^{0,\alpha}_{-2,\delta;t,\lambda}}\)
 \\
 &\leq c\(\|\rho_{t,\lambda}\ua\|_{C^{1,\alpha}_{-1,\delta;t,\lambda}}
         +\lambda^{2+\delta/2-\alpha}\|\pi_{t,\lambda}\ua\|_{C^{1,\alpha}}\)
\intertext{and}
  \|\pi_{t,\lambda}L_{t,\lambda}\ua\|_{C^{0,\alpha}}
 &\leq \|F_{\fI_t}\pi_{t,\lambda}\ua\|_{C^{0,\alpha}}
   +\|\pi_{t,\lambda}L_{t,\lambda}\rho_{t,\lambda}\ua\|_{C^{0,\alpha}} \\
  &\quad
   +c\lambda^{1-\alpha} \|\pi_{t,\lambda}\ua\|_{C^{0,\alpha}}
  \\
 &\leq c\(\|\pi_{t,\lambda}\ua\|_{C^{1,\alpha}}
         +\lambda^{-\alpha}\|\rho_{t,\lambda}\ua\|_{C^{1,\alpha}_{-1,\delta;t,\lambda}}\).
\end{align*}
This yields $\|\bL_{t,\lambda}\ua\|_{\fY_{t,\lambda}}\leq
c\|\ua\|_{\fX_{t,\lambda}}$.  In a similar way one shows that
\begin{equation*}
  \|\nabla_t\bL_{t,\lambda}\ua\|_{\fY_{t,\lambda}}\leq c\|\ua\|_{\fX_{t,\lambda}}.
\end{equation*}
This completes the proof.
\qed


\section{Quadratic estimate}
\label{Sec_QuadraticEstimate}

By a slight abuse of notation we denote by $Q_{t,\lambda}$ the
quadratic form defined in \eqref{Eq_QuadraticTerm} as well as the associated
bilinear form.

\begin{prop}\label{Prop_q}
  The bilinear form $Q_{t,\lambda}$ depends continuously
  differentiably on $t$ and continuously on $\lambda$ and there exists
  a constant $c>0$ such that for $t\in(-T',T')$ and
  $\lambda\in(0,\Lambda]$ we have
  \begin{align*}
    &\|\rho_{t,\lambda}Q_{t,\lambda}(\ua_1,\ua_2)\|_{C^{0,\alpha}_{-2,\delta;t,\lambda}} \\
    &\quad\leq c\lambda^{-\alpha}\(
    \|\rho_{t,\lambda}\ua_1\|_{C^{0,\alpha}_{-1,\delta;t,\lambda}}
    \cdot\|\rho_{t,\lambda}\ua_2\|_{C^{0,\alpha}_{-1,\delta;t,\lambda}}
    +\|\rho_{t,\lambda}\ua_1\|_{C^{0,\alpha}_{-1,\delta;t,\lambda}}
    \cdot\|\pi_{t,\lambda}\ua_2\|_{C^{0,\alpha}} \right. \\
    &\quad\qquad\qquad\left.+\|\pi_{t,\lambda}\ua_1\|_{C^{0,\alpha}}
    \cdot\|\rho_{t,\lambda}\ua_2\|_{C^{0,\alpha}_{-1,\delta;t,\lambda}}
    +\|\pi_{t,\lambda}\ua_1\|_{C^{0,\alpha}}
    \cdot \|\pi_{t,\lambda}\ua_2\|_{C^{0,\alpha}}\) \intertext{and}
    &\|\rho_{t,\lambda}\nabla_tQ_{t,\lambda}(\ua_1,\ua_2)\|_{C^{0,\alpha}_{-2,\delta;t,\lambda}} \\
    &\quad\leq c\lambda^{-\alpha}\(
    \|\rho_{t,\lambda}\ua_1\|_{C^{0,\alpha}_{-1,\delta;t,\lambda}}
    \cdot\|\rho_{t,\lambda}\ua_2\|_{C^{0,\alpha}_{-1,\delta;t,\lambda}}
    +\|\rho_{t,\lambda}\ua_1\|_{C^{0,\alpha}_{-1,\delta;t,\lambda}}
    \cdot\|\pi_{t,\lambda}\ua_2\|_{C^{0,\alpha}} \right. \\
    &\quad\qquad\qquad\left.+\|\pi_{t,\lambda}\ua_1\|_{C^{0,\alpha}}
    \cdot\|\rho_{t,\lambda}\ua_2\|_{C^{0,\alpha}_{-1,\delta;t,\lambda}}
    +\|\pi_{t,\lambda}\ua_1\|_{C^{0,\alpha}}
    \cdot \|\pi_{t,\lambda}\ua_2\|_{C^{0,\alpha}}\)
    \\ 
    \intertext{as well as}
    &\lambda\|\pi_{t,\lambda}Q_{t,\lambda}(\ua_1,\ua_2)\|_{C^{0,\alpha}} \\
    &\quad\leq c\lambda^{-\alpha}\(
    \|\rho_{t,\lambda}\ua_1\|_{C^{0,\alpha}_{-1,\delta;t,\lambda}}
    \cdot\|\rho_{t,\lambda}\ua_2\|_{C^{0,\alpha}_{-1,\delta;t,\lambda}}
    +\|\rho_{t,\lambda}\ua_1\|_{C^{0,\alpha}_{-1,\delta;t,\lambda}}
    \cdot\|\pi_{t,\lambda}\ua_2\|_{C^{0,\alpha}} \right. \\
    &\quad\qquad\qquad\left.+\|\pi_{t,\lambda}\ua_1\|_{C^{0,\alpha}}
    \cdot\|\rho_{t,\lambda}\ua_2\|_{C^{0,\alpha}_{-1,\delta;t,\lambda}}
    +\lambda \|\pi_{t,\lambda}\ua_1\|_{C^{0,\alpha}}
    \cdot \|\pi_{t,\lambda}\ua_2\|_{C^{0,\alpha}}\) 
    \intertext{and}
    &\lambda\|\pi_{t,\lambda}\nabla_tQ_{t,\lambda}(\ua_1,\ua_2)\|_{C^{0,\alpha}} \\
    &\quad\leq c\lambda^{-\alpha}\(
    \|\rho_{t,\lambda}\ua_1\|_{C^{0,\alpha}_{-1,\delta;t,\lambda}}
    \cdot \|\rho_{t,\lambda}\ua_2\|_{C^{0,\alpha}_{-1,\delta;t,\lambda}}
    +
    \|\rho_{t,\lambda}\ua_1\|_{C^{0,\alpha}_{-1,\delta;t,\lambda}}
    \cdot \|\pi_{t,\lambda}\ua_2\|_{C^{0,\alpha}} \right. \\
    &\quad\qquad\qquad\left.+\|\pi_{t,\lambda}\ua_1\|_{C^{0,\alpha}}
    \cdot\|\rho_{t,\lambda}\ua_2\|_{C^{0,\alpha}_{-1,\delta;t,\lambda}}
    +\lambda \|\pi_{t,\lambda}\ua_1\|_{C^{0,\alpha}}
    \cdot \|\pi_{t,\lambda}\ua_2\|_{C^{0,\alpha}}\).
  \end{align*}
\end{prop}

\begin{proof}
  The first two estimates are immediate consequences of  \autoref{Prop_Multiplication} and \autoref{Prop_IotaPi}.  For the last two estimates we only
  have to explain why we get a factor $\lambda$ (instead of one) in front of $\|\pi_{t,\lambda}\ua_1\|_{C^{0,\alpha}}
  \cdot \|\pi_{t,\lambda}\ua_2\|_{C^{0,\alpha}}$.  Note that
  \begin{equation*}
    \left[*_0\(\iota_{t,\lambda}\hat\fI_1\wedge\iota_{t,\lambda}\hat\fI_2\wedge\psi_{0;t}\)\right]^{0,1}=0
  \end{equation*}
  on grounds of simple bi-degree considerations.  Therefore, using
  \autoref{Prop_IotaPi}, \autoref{Prop_PsiExpansion} and \autoref{Prop_*Expansion},
  \begin{align*}
   &\|\pi_{t,\lambda}Q_{t,\lambda}(\iota_{t,\lambda}\hat\fI_1,\iota_{t,\lambda}\hat\fI_2)\|_{C^{0,\alpha}}  \\
&\quad\leq
 c\lambda^{-\alpha}\left(\|(*-*_0)(\iota_{t,\lambda}\hat\fI_1\wedge\iota_{t,\lambda}\hat\fI_2
   \wedge\psi)\|_{C^{0,\alpha}_{-1,\delta;t,\lambda}}\right. \\
&\quad\qquad\qquad+ \left.
 \|*_0[\iota_{t,\lambda}\hat\fI_1\wedge\iota_{t,\lambda}\hat\fI_2
   \wedge(\psi_{1;t}+\psi_{\geq
     2;t})]\|_{C^{0,\alpha}_{-1,\delta;t,\lambda}}\right) \\
&\quad\leq c\lambda^{-\alpha}\|\hat\fI_1\|_{C^{0,\alpha}}\cdot \|\hat\fI_2\|_{C^{0,\alpha}}.
\qedhere
  \end{align*}
\end{proof}


\section{Conclusion of the proof of \autoref{Thm_A}}
\label{Sec_Conclusion}

\begin{prop}\label{Prop_k}
  There is a constant $c>0$ and for $t\in(-T',T')$ and $\lambda \in
  (0,\Lambda]$ there are
  $\ua(t,\lambda)\in C^{1,\alpha}\(Y,(\Lambda^0\oplus\Lambda^1)\otimes\fg_{E_{t,\lambda}}\)$ and
  $\eta(t,\lambda) \in \R$ depending continuously differentiably on
  $t$ and continuously on $\lambda$ such that the connection $\tilde
  A_{t,\lambda}:=A_{t,\lambda} + a(t,\lambda)$ satisfies
  \begin{equation}\label{Eq_y}
    *\(F_{\tilde A_{t,\lambda}}\wedge\psi_t\) + \rd_{\tilde
      A_{t,\lambda}}\xi(t,\lambda) + \(\mu(t)+\eta(t,\lambda)\)\cdot\iota_{t,\lambda}\hat v\circ\fI_t = 0
  \end{equation}
  and
  \begin{gather*}
    \|\ua(t,\lambda)\|_{\fX_{t,\lambda}}\leq c
    \lambda^{2-\alpha} \quad\text{and}\quad
    |\eta(t,\lambda)| + |\del_t\eta(t,\lambda)| \leq c \lambda^{1-\alpha}.
  \end{gather*}
\end{prop}

The proof relies on the preceding analysis and the following simple
consequence of Banach's fixed point theorem,
cf.~\cite[Lemma~7.2.23]{Donaldson1990}.
\begin{lemma}\label{Lem_CM}
  Let $X$ be a Banach space and let $T \co X\to X$ be a smooth map
  with $T(0)=0$.  Suppose there is a constant $c>0$ such that
    \begin{equation*}
      \|Tx-Ty\|\leq c\(\|x\|+\|y\|\)\|x-y\|.
    \end{equation*}
    Then if $y\in X$ satisfies $\|y\|\leq \frac{1}{10c}$, there exists
    a unique $x\in X$ with $\|x\|\leq \frac{1}{5c}$ solving
    \begin{equation*}
      x+Tx=y.
    \end{equation*}
    The unique solution satisfies $\|x\|\leq 2\|y\|$.  Moreover, if
    $T$ and $y$ depend continuously or continuously differentiably on
    a parameter in an open subset of $\R^n$, then so does the solution
    $x$.
  \end{lemma}

\begin{proof}[Proof of \autoref{Prop_k}]
  We solve \eqref{Eq_y} with the additional
  constraints
  \begin{equation*}
    \rd_{A_{t,\lambda}}^*a=0 \qandq
    \<\pi_{t,\lambda}\ua,\hat v\circ \fI_t\>=0.
  \end{equation*}
  This can be written as
  \begin{equation*}
    \bL_{t,\lambda} (\ua,\eta) + Q_{t,\lambda}(\ua) +
    e_{t,\lambda} = 0.
  \end{equation*}
  With $(\ua,\eta)=\bL_{t,\lambda}^{-1}(\ub,\zeta)$ this becomes
  \begin{equation}\label{Eq_fpe}
    (\ub,\zeta) + \tilde Q_{t,\lambda}(\ub,\zeta) +
    e_{t,\lambda} = 0.
  \end{equation}
  where $\tilde Q_{t,\lambda}:=Q_{t,\lambda}\circ
  \bL_{t,\lambda}^{-1}$.  It follows from
  \autoref{Prop_inverse} and \autoref{Prop_q} that
  \begin{align*}
    &\|\tilde Q_{t,\lambda}(\ub_1,\zeta_1)-\tilde
    Q_{t,\lambda}(\ub_2,\zeta_2)\|_{\fY_{t,\lambda}} \\
    &\qquad\leq c \lambda^{-2-\delta/2-\alpha}
    \(\|(\ub_1,\zeta_1)\|_{\fY_{t,\lambda}}
    +\|(\ub_2,\zeta_2)\|_{\fY_{t,\lambda}}\)
    \|(\ub_1,\zeta_1)-(\ub_2,\zeta_2)\|_{\fY_{t,\lambda}}.
  \end{align*}
  and we recall from \autoref{Prop_PregluingEstimate} that
  \begin{equation*}
    \|e_{t,\lambda}\|_{\fY_{t,\lambda}}\leq c\lambda^{2-\alpha}.
  \end{equation*}
  Hence, we can solve \eqref{Eq_fpe} using \autoref{Lem_CM} since
  $\delta \in(-1,0)$ and $0<\alpha\ll |\delta|$.  The solution
  satisfies $\|(\ub,\zeta)\|_{\fY_{t,\lambda}}\leq
  c\lambda^{2-\alpha}$ and $(\nabla_t\ub,\del_t\zeta)$ solves the
  equation
  \begin{equation}
    (\nabla_t\ub,\del_t\zeta) +
    2\tilde Q_{t,\lambda}((\ub,\zeta),(\nabla_t\ub,\del_t\zeta))+
     (\nabla_t \tilde Q_{t,\lambda})(\ub,\zeta) + \nabla_t e_{t,\lambda}=0.
  \end{equation}
  Since $\|2\tilde Q_{t,\lambda}(\ub,\cdot)\|_{\fY_{t,\lambda}}\leq
  \frac12\|\cdot\|_{\fY_{t,\lambda}}$ and $\|(\nabla_t \tilde
  Q_{t,\lambda})(\ub,\zeta) + \nabla_t
  e_{t,\lambda}\|_{\fY_{t,\lambda}}\leq c\lambda^{2-\alpha}$, it follows that
  $\|(\nabla_t\ub,\del_t\zeta)\|_{\fY_{t,\lambda}}\leq c\lambda^{2-\alpha}$.  This
  implies the desired estimates on
  $(\ua,\eta)=\bL_{t,\lambda}^{-1}(\ub,\zeta)$ and its derivative by
  \autoref{Prop_inverse}.
\end{proof}

The problem of finding $\bar A_\lambda$ is now reduced to constructing a continuous function
$t\co[0,\Lambda]\to(-T',T')$ such that $t(0)=0$ and
\begin{equation*}
  \mu(t(\lambda))+\eta(t(\lambda),\lambda) = 0
\end{equation*}
for $t\in(0,\Lambda]$.
Since $\mu(0)=0$ and $\del_t \mu(0)\neq 0$, we can invert $\mu$ locally around $t = 0$ and rewrite this equation as
\begin{equation}
  \label{Eq_ResidualTransformed}
  \tilde\mu(\lambda)+\eta(\mu^{-1}\circ\tilde\mu(\lambda),\lambda) = 0
\end{equation}
with $\tilde\mu = \mu \circ t$.
Because $|\eta|+|\del_t\eta|\leq c\lambda^{1-\alpha}$, this equation on the other hand can immediately be solved for $\tilde\mu$ and thus $t = \mu^{-1}\circ\tilde \mu$ by appealing to \autoref{Lem_CM}.

\begin{remark}
  If we assume the situation of \autoref{Rmk_WeaklyUnobstructed}, that is, $\mu$ is just monotone (but possibly $\del_t \mu(0) = 0$), then one can still find a continuous inverse $\mu^{-1}$ find solutions of \eqref{Eq_ResidualTransformed} using Brouwer's fixed point theorem.
  However, these solutions might not be described by the graph of a function;
  e.g., if $\mu(t) = t^3$ and $\eta(t,\lambda) = -t\lambda^2$, then the set of solutions of $\mu(t) + \eta(t,\lambda) = 0$ is a union of three graphs: $t = 0$ and $t = \pm \lambda$.
\end{remark}

The resulting connection $A_\lambda:=\tilde A_{t(\lambda),\lambda}$ will be smooth by elliptic regularity.
That $A_\lambda$ converges to $B_0$ on the complement of $P_0$ and that at each point $x\in P_0$ an ASD instanton in the equivalence class of $\fI(x)$ bubbles off transversely is clear, since we constructed $A_{t,\lambda}$ accordingly and $\bar A_\lambda$ is a small perturbation of $A_{t,\lambda}$.
This concludes the proof of \autoref{Thm_A}.
\qed


\printreferences

\end{document}
